\documentclass[twoside, 11pt]{article}
\usepackage[english]{babel}
\usepackage[latin1]{inputenc}
\usepackage{amsmath,amssymb}
\usepackage{amsfonts}
\usepackage{latexsym}
\usepackage{graphics,graphicx,psfrag}
\usepackage{upref}
\usepackage{hyperref}
\usepackage{caption,accents, float}
\usepackage{cite}
\usepackage[dvipsnames]{color, xcolor}
\usepackage{changes}

\allowdisplaybreaks[1]

\numberwithin{equation}{section}

\newcommand{\calk}{\mathcal{K}}
\newcommand{\partt}{\frac{\partial}{\partial t}}

\newcommand{\supp}{\mathop{\mathrm{supp}\,}\nolimits}

\newcommand{\real}{\mathbb{R}}

\newcommand{\nat}{\mathbb{N}}

\newcommand{\ud}{\mathrm{d}}
\newcommand{\uD}{\mathrm{D}}

\newtheorem{lemma}{Lemma}[section]

\newtheorem{theorem}[lemma]{Theorem}
\newtheorem{corollary}[lemma]{Corollary}
\newtheorem{definition}[lemma]{Definition}
\newtheorem{assumption}[lemma]{Assumption}
\newtheorem{rem}[lemma]{Remark}
\newcommand{\remark}[1]{\begin{rem}{\upshape #1}\end{rem}}
\newtheorem{example}[lemma]{Example}

\newcommand{\proofstart}{\mbox{P\,r\,o\,o\,f\, :\quad}}
\newcommand{\proofend}{\nopagebreak\hfill\raisebox{0.3em}{\fbox{}}\\}
\newenvironment{proof}{\proofstart}{\proofend}

\newtheorem{problems}{Problem}

\newcommand{\D}{\mathcal{D}}
\newcommand{\V}{\mathcal{K}}

\newcommand{\beq}{\begin{equation}}
\newcommand{\eeq}{\end{equation}}

\begin{document}

\title{Regularity in Sobolev and Besov spaces for parabolic problems on domains of polyhedral type}

\author{
Stephan Dahlke\footnote{Philipps-University Marburg,  FB12 Mathematics and Computer Science, Hans-Meerwein Stra\ss{}e, Lahnberge, 35032 Marburg, Germany. Email: \href{mailto:dahlke@mathematik.uni-marburg.de}{dahlke@mathematik.uni-marburg.de}}
\qquad
Cornelia Schneider\footnote{\emph{Corresponding author}. Friedrich-Alexander University Erlangen-Nuremberg, Applied Mathematics III, Cauerstr. 11, 91058 Erlangen, Germany. Email: \href{mailto:cornelia.schneider@math.fau.de}{cornelia.schneider@math.fau.de}}\ \thanks{The work of this author has been supported by Deutsche Forschungsgemeinschaft (DFG), Grant No. SCHN 1509/1-2.}
}

\date{}
\maketitle

\begin{abstract}
This paper is concerned with the regularity of  solutions {to} linear and nonlinear evolution equations extending our findings in \cite[Thms.~4.5,4.9,4.12,4.14]{DS19} to  domains of polyhedral type.   In particular, we study the smoothness in the specific scale 
$\ B^r_{\tau,\tau}, \ \frac{1}{\tau}=\frac{r}{d}+\frac{1}{p}\ $ of Besov spaces.   The regularity in these spaces determines the approximation order
that can be achieved by adaptive and other nonlinear approximation schemes. We show that for all cases under consideration the Besov regularity is high enough to 
justify the use of adaptive algorithms. 

{\em Key Words:} Parabolic evolution equations, Besov spaces, Kondratiev spaces, adaptive algorithms. \\
{\em Math Subject Classifications. Primary:}   35B65,  35K55, 46E35.   {\em Secondary:} 35L15, 35A02,  35K05,  65M12.

\end{abstract}


\section{Introduction}\label{introduction}

This paper is concerned with regularity estimates of the solutions to
 evolution equations in non-smooth  domains of polyhedral type $D \subset \real^3$, cf. Definition \ref{standard}.
In particular, we study   linear ($\varepsilon=0$) and nonlinear  ($\varepsilon>0$) equations of the form 
\begin{equation} \label{parab-1a-i}
\frac{\partial }{\partial t}u+(-1)^mL(t,x,D_x)u +\varepsilon u^{M}\ =\ f \quad    \text{ in }\  [0,T]\times D, 
\end{equation}
 with zero initial and Dirichlet boundary conditions, where  $m,M\in \nat$,  and $L$ denotes a uniformly elliptic operator of order $2m$ with sufficiently smooth coefficients. 
 Special attention is paid to the spatial regularity of the
solutions to  (\ref{parab-1a-i})  in specific
non-standard  smoothness spaces, i.e., in the so-called {\em adaptivity
scale {of Besov spaces}}
\begin{equation} \label{adaptivityscale}
B^r_{\tau,\tau}(D), \quad \frac{1}{\tau}=\frac{r}{3}+\frac{1}{p}, \quad r>0.
\end{equation}

Our investigations are motivated by  fundamental questions arising in the context of the numerical treatment of equation  \eqref{parab-1a-i}. In particular, we aim at justifying the use of adaptive numerical methods for parabolic PDEs. 
Let us explain these relationships in more detail:   In an adaptive strategy, the choice of the underlying degrees of freedom is not a priori fixed but depends on the shape of the unknown solution. In particular, additional degrees of freedom are only spent in regions where the numerical approximation is still 'far away' from the exact solution.
Although the basic idea is convincing,  adaptive algorithms are hard to implement, so that beforehand a rigorous mathematical analysis to justify their use  is highly desirable. 

Given an adaptive algorithm based on a dictionary for the solution spaces of the PDE, the best one can expect is an optimal performance in the sense that it realizes the convergence rate of best $N$-term approximation schemes, which serves as a benchmark in this context.  Given a dictionary
 $\Psi =\{\psi_{\lambda}\}_{\lambda \in \Lambda}$  of functions in a Banach space $X$,
the error of best $N$-term approximation is defined as
\begin{equation}\label{error-n-term}
\sigma_N\bigl(u;X\bigr)=\inf_{\Gamma\subset\Lambda:\#\Gamma\leq
N}\inf_{c_\lambda}
                \biggl\|u-\sum_{\lambda \in\Gamma}c_{\lambda}\psi_{\lambda}\big|X\biggr\|\, , 
\end{equation}
i.e., as the name suggests we consider the best approximation by linear combinations of the basis functions consisting of at most $N$ terms. 
 In particular, \cite[Thm. 11, p.~586]{DNS2} implies for $\tau<p$, 
\begin{equation*}
    \sigma_N\bigl(u;L_p(D)\bigr)\leq C\,N^{-s/d}\|u|B^s_{\tau,\tau}(D)\|, \qquad  \frac{1}{\tau}<\frac{s}{3}+\frac 1p. 
\end{equation*}
Quite recently, it has turned out that the same interrelations also hold for the very important and  widespread adaptive finite
 element schemes. In particular,  \cite[Thm. 2.2]{GM09} gives direct estimates,
\begin{equation*}
    \sigma_N^{FE}\bigl(u;L_p(D)\bigr)\leq C\,N^{-s/d}\|u|B^s_{\tau,\tau}(D)\|\,,  
\end{equation*}
 where $\sigma_N^{FE}$ denotes the counterpart to the quantity $\sigma_N(u;X)$, which corresponds to wavelet approximations. It can be seen that 
the achievable order of adaptive algorithms depends on the regularity of the target function in the specific scale of Besov spaces \eqref{adaptivityscale}. On the other hand it is the regularity of the solution in the scale of Sobolev spaces, which encodes information on the convergence order for nonadaptive (uniform) methods.  
From this we can draw the following conclusion: adaptivity is justified,
 if the Besov regularity of the solution
in the Besov scale
(\ref{adaptivityscale}) is higher than its Sobolev smoothness!

For the case of {\em elliptic} partial differential equations, a lot of positive results in this direction are already established  \cite{Dah98, Dah99a, Dah99b,Dah02,DDD,DDV97,DDHSW,Han15,HW18}. It is well--known
that  if the domain under consideration, the right--hand side and the coefficients are sufficiently smooth, then the problem is completely regular \cite{ADN59}, and there
is no reason why the  Besov smoothness should be higher than the Sobolev regularity. However, on general Lipschitz domains and in particular in polyhedral domains, the situation changes
dramatically.  On  these domains, singularities at the boundary may occur that diminish the Sobolev regularity of the solution significantly \cite{CW20,Cost19,JK95, Gris92, Gris11}.
However, the analysis in the above mentioned papers shows that these boundary singularities do not influence the Besov regularity too much, so that the use of
adaptive algorithms for elliptic PDEs is completely justified!

 In this paper, we study similar questions for evolution equations of the form  
(\ref{parab-1a-i}) and of associated 
semilinear versions. 
To the best of our knowledge, not so many results in this direction are available
so far.  For parabolic equations, first  results   for the special case of the heat equation have been reported 
in \cite{AGI08, AGI10, AG12}, but for a slightly different scale of Besov spaces.

Our results show in the linear case $\varepsilon=0$ that if the right-hand side as well as its time derivatives are
contained in specific Kondratiev spaces, then, for every $t \in [0,T]$  the spatial Besov smoothness  
of the solution to \eqref{parab-1a-i} is always larger than  $2m$, provided that some technical conditions on the operator pencils are satisfied, see Theorems \ref{thm-parab-Besov} and \ref{thm-parab-Besov-2}.
The reader should observe that the results are independent of the shape of the polyhedral domain, and that the classical Sobolev smoothness
is usually limited by  $m$, see \cite{LL15}.  Therefore, for every $t$, the spatial Besov regularity is more than twice as
high as the Sobolev smoothness, which of course justifies the use of (spatial) adaptive algorithms. 
Moreover,  for smooth domains
and  right-hand sides in $L_2,$ the best one could  expect would be smoothness order $2m$ in 
the classical Sobolev scale. So, the Besov smoothness on polyhedral type domains is at least as high as the Sobolev smoothness 
on smooth domains. \\
Afterwards, we generalize this result to nonlinear parabolic  equations of the form \eqref{parab-1a-i}. We show that in a sufficiently small ball containing the solution of the corresponding linear equation, there exists
a unique solution to  
\eqref{parab-1a-i} possessing the same Besov smoothness in the scale 
\eqref{adaptivityscale}.  The proof is performed by a technically quite involved application of the Banach fixed point theorem. 
The final result is stated in Theorem \ref{nonlin-B-reg3}. \\
The next natural step is to also  study  the regularity in time direction.  For the linear parabolic problem \eqref{parab-1a-i} with $\varepsilon=0$ we show that the mapping  $t\mapsto u(t, \cdot)$  is in fact a $C^l$-map into the adaptivity scale
 of Besov spaces,  
precisely, 
$$u \in \mathcal{C}^{l,\frac{1}{2}}((0,T), B^{\alpha}_{\tau,\infty}(D)), $$
see Theorem \ref{Hoelder-Besov-reg}.

In conclusion, the results presented in this paper imply that for each $t \in (0,T)$ the spatial Besov regularity of the unknown solutions of the problems studied here is much higher than the Sobolev regularity,  which justifies the use of spatial adaptive  algorithms.  This corresponds to the classical time-marching schemes such as the Rothe method. We refer e.g. to the monographs \cite{Lan01, Tho06} for a detailed discussion. Of course,  it would be tempting to employ adaptive strategies in the whole space-time cylinder.  First results in this direction have been reported in \cite{SS09}.  To justify also these schemes,  Besov regularity in the whole space-time cylinder has to be established. This case will be studied in a forthcoming paper.\\ 
Throughout the paper we use the same notation as in \cite{DS19}, which for the convenience of the reader is  recalled in Appendix \ref{app-not}.

\section{Sobolev and Kondratiev spaces}\label{Sect-2}

In this section we briefly collect the  basics  concerning   weighted and unweighted Sobolev  spaces needed later on. In particular, we put $H^m=W^m_2$ and denote by $\mathring{H}^m$ the closure of test functions in $H^m$ and its  dual space by $H^{-m}$. 
Moreover,  $\mathcal{C}^{k,\alpha}$, $k\in \nat_0$, stands for the usual H\"older spaces with exponent $\alpha\in (0,1]$. 
The following generalized version of Sobolev's embedding theorem for Banach-space valued functions will be useful, cf. \cite[Thm.~1.2.5]{co-habil}.

\begin{theorem}[{\bf Generalized Sobolev's embedding theorem}]\label{thm-sob-emb}\index{Sobolev's embedding theorem! generalized}
Let $1<p<\infty$,  $m\in \nat$, $I\subset \real$ be some bounded interval, and $X$ a Banach space. Then 
\begin{equation}
W^{m}_p(I,X)\hookrightarrow \mathcal{C}^{m-1,1-\frac 1p}(I,X). 
\end{equation}
\end{theorem}

Here the Banach-valued  Sobolev spaces are endowed with the  norm 
\[
\|u|W^{m}_p(I,X)\|^p:=\sum_{k=0}^m \|\partial_{t^k}u|L_p(I,X)\|^p \quad \text{with}\quad \|\partial_{t^k}u|L_p(I,X)\|^p:=\int_I \|\partial_{t^k}u(t)|X\|^p~\ud t,  
\]
whereas for the H\"older spaces we use 
\[
\|u|\mathcal{C}^{k,\alpha}(I,X)\|:=\|u|C^k(I,X)\|+|u^{(k)}|_{C^{\alpha}(I,X)},
\]
where $\|u|C^k(I,X)\|=\sum_{j=0}^k\max_{t\in I}\|u^{(j)}(t)|X\|$ and $|u^{(k)}|_{C^{\alpha}(I,X)}=\sup_{{s,t\in I,}\atop  {s\neq t}}\frac{\|u^{(k)}(t)-u^{(k)}(s)|X\|}{|t-s|^{\alpha}}$.

We collect some notation for specific Banach-space valued Lebesgue and Sobolev spaces, which will be used when studying the regularity of solutions of parabolic PDEs. \\

Let   $\Omega_T:=[0,T]\times\Omega$. Then we abbreviate   
\[
L_p(\Omega_T):=L_p([0,T], L_p(\Omega)).    
\]\label{extra-1}

Moreover, we put \label{extra-4}
\[
H^{m,l*}(\Omega_T):=H^{l-1}([0,T],\mathring{H}^m(\Omega))\cap H^l([0,T],H^{-m}(\Omega)) 
\]
normed by 
\[\|u|H^{m,l*}(\Omega_T)\|=\|u|H^{l-1}([0,T],\mathring{H}^m(\Omega))\|+\|u|H^l([0,T],H^{-m}(\Omega))\|.\]

\subsection{Kondratiev spaces}

In the sequel we work to a great extent with   weighted Sobolev spaces,  the so-called {\em Kondratiev spaces} $\V^m_{p,a}(\mathcal{O})$,   defined as the collection of all  $u\in \mathcal{D}'(\mathcal{O})$, which have $m$ generalized derivatives satisfying 

\begin{equation}\label{Kondratiev-1}
\|u|\V^m_{p,a}(\mathcal{O})\|:=\left(\sum_{|\alpha|\leq m}\int_{\mathcal{O}} |\varrho(x)|^{p(|\alpha |-a)}|D^{\alpha}_x u(x)|^p\ud x\right)^{1/p}<\infty,
\end{equation}
where $a\in \real$, $1<p<\infty$, $m\in \nat_0$, $\alpha\in \nat^n_0$, and the weight function $\varrho: D\rightarrow [0,1]$ is the smooth distance to the singular set of $\mathcal{O}$, i.e., $\varrho$ is a smooth function and in the vicinity of the singular set $S$  it is {equivalent} to the distance to that set.  Clearly, if $\mathcal{O}$ is a polygon in $\real^2$ or a  polyhedral domain in $\real^3$, then  the  singular set  $S$ consists of the vertices of the polygon or the vertices and edges of the polyhedra, respectively.

It follows directly from \eqref{Kondratiev-1} that the scale of Kondratiev spaces is monotone in $m$ and $a$, i.e., 
\beq\label{kondratiev-emb}
\V^m_{p,a}(\mathcal{O})\hookrightarrow \V^{m'}_{p,a}(\mathcal{O})\quad \text{and}\quad \V^m_{p,a}(\mathcal{O})\hookrightarrow \V^m_{p,a'}(\mathcal{O}),
\eeq 
if $m'<m$ and $a'<a$.

Moreover, generalizing the above concept to functions depending on the time $t\in [0,T]$, we define Kondratiev type spaces,    denoted by $L_q((0,T),\V^m_{p,a}(\mathcal{O}))$, which   contain all functions $u(x,t)$ such that 
\begin{align}
\|u|&L_q((0,T), \V^m_{p,a}(\mathcal{O}))\|\notag\\
&:=\left(\int_{(0,T)}\left(\sum_{|\alpha|\leq m}\int_{\mathcal{O}} |\varrho(x)|^{p(|\alpha |-a)}|D^{\alpha}_x u(x,t)|^p\ud x\right)^{q/p}\ud t\right)^{1/q}<\infty, \label{Kondratiev-3}
\end{align}
with $0<q\leq \infty$ and  parameters $a,p,m$  as above.

\paragraph{Kondratiev spaces on domains of polyhedral type}
\label{domains}

For our analysis we make use of several properties of Kondratiev spaces that have been proved in \cite{DHS17a}. 
Therefore,   in our later considerations, we will mainly be interested in the case that   $\mathcal{O}$ is a bounded domain of polyhedral type.

The precise definition below is taken from  Maz'ya and Rossmann \cite[Def.~4.1.1]{MR10}.

\begin{definition}\label{standard}
A bounded domain $D\subset \real^3$ is defined to be  of polyhedral type if the following holds: \\
\begin{figure}[H]
\begin{minipage}{0.55\textwidth}
\begin{itemize}
\item[(a)] The boundary $\partial D$ consists of smooth (of class $C^{\infty}$) open two-dimensional manifolds $\Gamma_j$ (the faces of $D$), $j=1,\ldots, n$, smooth curves $M_k$ (the edges), $k=1,\ldots, l$, and vertices $x^{(1)}, \ldots, x^{(l')}$. 
\item[(b)] For every $\xi\in M_k$ there exists a neighbourhood $U_{\xi}$ and a $C^{\infty}$-diffeomorphism  $\kappa_{\xi}$ which maps $D\cap U_{\xi}$ onto $\mathcal{D}_{\xi}\cap B_1(0)$, where $\mathcal{D}_{\xi}\subset \real^3$ is a dihedron, which in polar coordinates can be described as 
 \end{itemize}
\end{minipage}\hfill 
\begin{minipage}{0.25\textwidth}
\includegraphics[width=4cm]{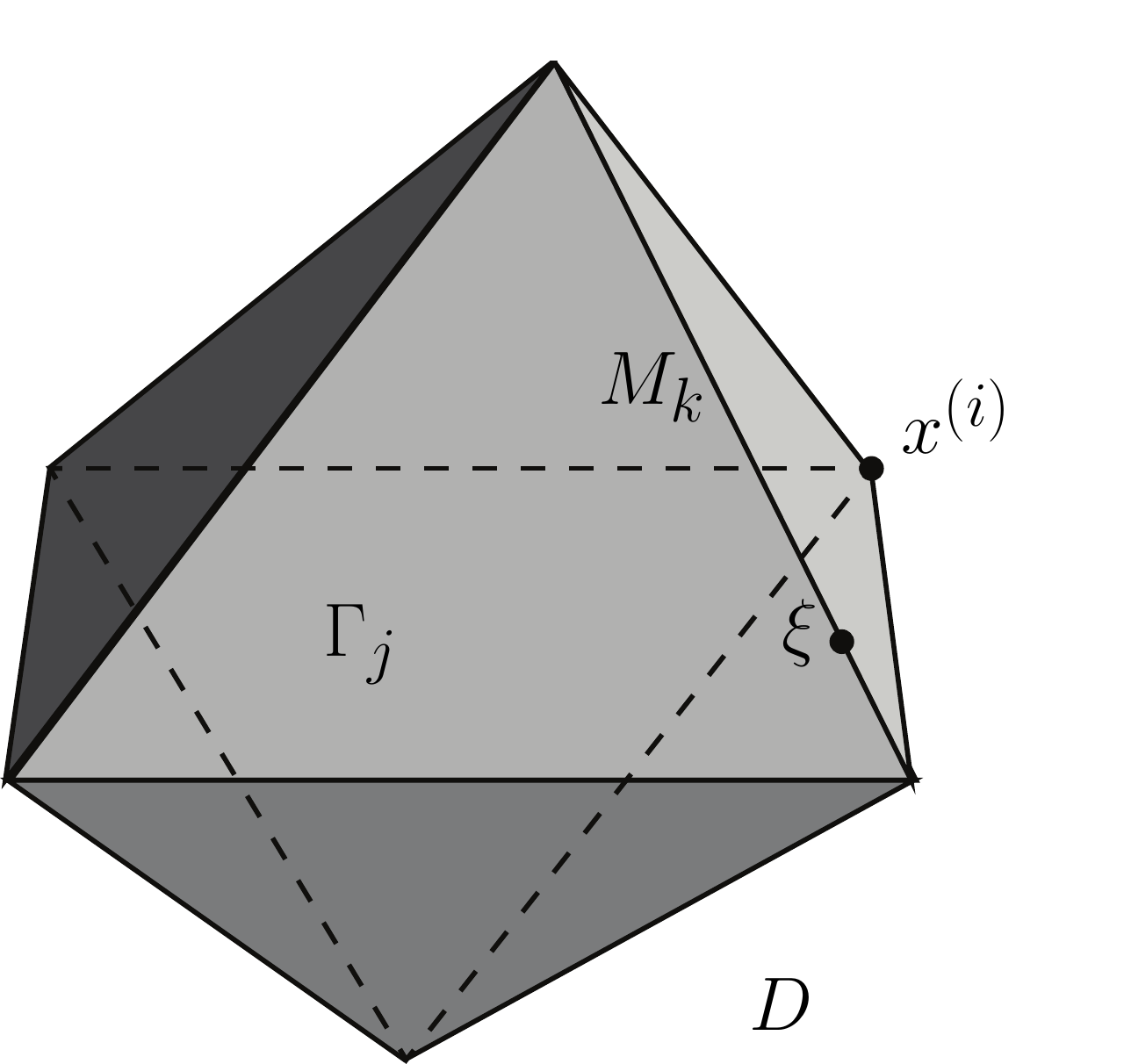}
\caption[Polyhedron]{Polyhedron}
\end{minipage}\\
\begin{itemize}
    \item[] 
\[
\mathcal{D}_{\xi}=K\times \real, \qquad K=\{(x_1,x_2): \ 0<r<\infty, \ -\theta/2<\varphi<\theta/2\},
\]
where the opening angle $\theta$ of the 2-dimensional wedge $K$ satisfies $0<\theta\leq 2\pi$. 
\item[(c)] For every vertex $x^{(i)}$ there exists a neighbourhood $U_i$ and a diffeomorphism $\kappa_i$ mapping $D\cap U_i$
 onto $K_i\cap B_1(0)$, where $K_i$ is a polyhdral cone with edges and vertex at the origin. 
 \end{itemize}
\end{figure}

\end{definition}

\begin{rem}\label{rem-def-domain} \label{notation}
\begin{itemize}
\item[(i)]
In the literature many different types of polyhedral domains are considered. A more general version which coincides with the above definition when $d=3$ is discussed in \cite{DHS17a}.
Further variants of polyhedral domains can be found in 
Babu\v{s}ka,  Guo \cite{BG97},  Bacuta, Mazzucato, Nistor, Zikatanov \cite{BMNZ} 
and Mazzucato, Nistor \cite{NistorMazzucato}.
\item[(ii)] Let us point out that 'smooth' domains without edges and/or vertices are admissible in Defintion \ref{standard}. We discuss this further in Section \ref{subsect-op-pen}. 
\end{itemize}
\end{rem}

\smallskip

\paragraph{Some properties of Kondratiev spaces}

Concerning pointwise multiplication the following results are proven in \cite{DHS17a}.

\begin{corollary}\label{thm-pointwise-mult-2}
\begin{itemize}
\item[(i)] Let $m\in \nat$, $a\geq \frac 3p$, and either $1<p<\infty$ and $m>\frac 3p$ or $p=1$ and $m\geq 3$. 
Then the Kondratiev space $\calk^m_{a,p}(D)$ is an algebra with respect to pointwise multiplication, i.e.,  there exists a constant $c$ such that 
\[
\|uv| \mathcal{K}^{m}_{a,p}(D)\|\leq c\|u|\mathcal{K}^{m}_{a,p}(D)\|\cdot \|v|\mathcal{K}^{m}_{a,p}(D)\|
\]
holds for all $u,v\in \mathcal{K}^{m}_{a,p}(D)$.
\item[(ii)] Let $\frac{3}{2}<p<\infty$, $m\in \nat$, and $a\geq \frac{3}{p}-1$. Then there exists a constant $c$ such that 
\[
\|uv| \mathcal{K}^{m-1}_{a-1,p}(D)\|\leq c\|u|\mathcal{K}^{m+1}_{a+1,p}(D)\|\cdot \|v|\mathcal{K}^{m-1}_{a-1,p}(D)\|
\]
holds for all $u\in \mathcal{K}^{m+1}_{a+1,p}(D)$ and $v\in \mathcal{K}^{m-1}_{a-1,p}(D)$.
\end{itemize}
\end{corollary}

Our main tool when investigating the Besov regularity of  solutions to the PDEs will be the following embedding result  between  Kondratiev  and Besov spaces, which is an extension of  \cite[Thm.~1]{Han15}. A proof may be found in \cite[Thm.~1.4.12]{co-habil}.

\begin{theorem}[{\bf Embeddings between Kondratiev and Besov spaces}]\label{thm-hansen-gen}
Let $D\subset \real^3$ be some  polyhedral type domain and assume $k\in \nat_0$, $0<q\leq \infty$. Furthermore, let $s, a\in \real$, $\gamma\in \nat_0$, and  suppose $\min(s,a)>\frac{\delta}{3}\gamma$, where $\delta$ denotes the dimension of the singular set (i.e., $\delta=0$ if there are only vertex singularities and $\delta=1$ if there are edge and vertex singularities). Then there exists some $0<\tau_0\leq p$ such that 
{\begin{equation}\label{emb-hansen-gen-sob}
W^k_q([0,T],\calk^{\gamma}_{p,a}(D))\cap W^k_q([0,T],B^s_{p,\infty}(D))\hookrightarrow W^k_q([0,T],B^{\gamma}_{\tau,\infty}(D))  
\end{equation}}
for all $\tau_{\ast}<\tau<\tau_0$, where $\frac{1}{\tau_{\ast}}=\frac {\gamma}{3}+\frac 1p$. 
\end{theorem}

\section{Parabolic PDEs and operator pencils}  
\label{sect-fund-prob}

In the sequel we deal with two different parabolic settings,  Problems \ref{prob_parab-1a} and \ref{prob_nonlin}, which are of general order and defined on domains of polyhedral type according to Definition \ref{standard}. In particular, Problem \ref{prob_nonlin} is the nonlinear version of Problem \ref{prob_parab-1a} and we investigate the spatial Besov regularity of the solutions of these two problems and to some extent also the  H\"older regularity with respect to the time variable of Problem \ref{prob_parab-1a}. \\

\subsection{The fundamental parabolic problems}

Let  $D$  denote some domain of polyhedral type in $\real^d$ according to Definition \ref{standard} with faces $\Gamma_j$, $j=1,\ldots, n$.  
For $0<T<\infty$ put $D_T=(0,T]\times D$ and 
$ 
\Gamma_{j,T}=[0,T]\times \Gamma_j$.  \\

We will investigate the Besov regularity of the following linear  parabolic problem.

\begin{problems}[{\bf Linear parabolic problem in divergence form}]\label{prob_parab-1a}
Let $m\in \nat$. We consider the following first initial-boundary value problem 

\begin{equation} \label{parab-1a}
\left\{\begin{array}{rl}
\partt u+(-1)^m{L(t,x,D_x)}u\ =\ f \, &  \text{ in } D_T, \\
\frac{\partial^{k-1}u}{\partial \nu^{k-1}}\Big|_{\Gamma_{j,T}}\ =\ 0, & \   k=1,\ldots, m, \ j=1,\ldots, n,\\ 
u\big|_{t=0}\ =\ 0 \, & \text{ in } D.
\end{array} \right\}
\end{equation}
\end{problems}

Here {$f$ is a function given on $D_T$, $\nu$ denotes the exterior normal to $\Gamma_{j,T}$}, and  the partial differential operator $L$ is given by
\[{L(t,x,D_x)}=\sum_{|\alpha|, |\beta|=0}^m D^{\alpha}_x({a_{\alpha \beta}(t,x)}D^{\beta}_x),\]
where $a_{\alpha \beta}$ are bounded real-valued functions from $C^{\infty}(D_T)$ with  $a_{\alpha \beta}=(-1)^{|\alpha|+|\beta|}{a}_{\beta \alpha}$. 
Furthermore, the operator $L$ is assumed to be uniformly elliptic  with respect to $t\in [0,T]$, i.e., 
\begin{equation}\label{operator_L}
\sum_{|\alpha|, |\beta|=m}a_{\alpha \beta}\xi^{\alpha}\xi^{\beta}\geq c|\xi|^{2m} \qquad {\text{for all}}\quad  (t,x)\in D_T, \quad \xi\in \mathbb{R}^d.
\end{equation}

 Let us denote by 
\begin{equation}
B(t,u,v)=\int_D \sum_{|\alpha|, |\beta|=0}^m a_{\alpha\beta}(t,x)(D^{\beta}_xu) (D^{\alpha}_xv)\ud x
\end{equation}
the time-dependent bilinear form.

Moreover, for simplicity we set 
\begin{equation}\label{deriv-B}
B_{\partial_{t^{k}}}(t,u,v)=\sum_{|\alpha|, |\beta|\leq m}\int_D\frac{\partial a_{\alpha\beta}(t,x)}{\partial t^k}(\uD^{\beta}_xu)(t,x)(\uD^{\alpha}_xv)(t,x)\ud x.
\end{equation}

\begin{rem}[{\bf Assumptions on the time-dependent bilinear form}]\label{rem-B-coercive}
When dealing with parabolic problems  it will be reasonable to  suppose that $B(t,\cdot, \cdot)$ satisfies 
\begin{equation}\label{B-coercive}
B(t,u,u)\geq \mu \|u|H^m(D)\|^2
\end{equation}
for all $u\in \mathring{H}^m(D)$ and a.e. $t\in [0,T]$. 
We refer to \cite[Rem.~2.3.5]{co-habil} for a detailed discussion. 
\end{rem}

It is our intention to also  study nonlinear  versions of Problem \ref{prob_parab-1a}. Therefore, we  modify \eqref{parab-1a} as  follows. 

\begin{problems}[{\bf Nonlinear parabolic problem in divergence form}]\label{prob_nonlin}
Let $m,M\in \nat$  and $\varepsilon>0$. We consider the following nonlinear parabolic problem 
\begin{equation} \label{parab-nonlin-1}
\left\{\begin{array}{rl}
\frac{\partial }{\partial t}u+(-1)^mL(t,x,D_x)u +\varepsilon u^{M}\ =\ f \, &  \text{ in } D_T, \\
\frac{\partial^{k-1}u}{\partial \nu^{k-1}}\Big|_{\Gamma_{j,T}}\ =\ 0, & \   k=1,\ldots, m, \ j=1,\ldots, n,\\ 
u\big|_{t=0}\ =\ 0 \, & \text{ in } D. 
\end{array} \right\}
\end{equation}
\end{problems}
The assumptions on $f$ and  the operator $L$ are as in Problem \ref{prob_parab-1a}. When we establish Besov regularity results for Problem \ref{prob_nonlin} we interpret  \eqref{parab-nonlin-1} as a fixed point problem and  show that the regularity estimates for Problem \ref{prob_parab-1a}   carry over to Problem \ref{prob_nonlin}, provided that $\varepsilon$ is sufficiently small.

\subsection{Operator pencils} 
\label{subsect-op-pen}

In order to correctly state the global regularity results in Kondratiev spaces for Problems \ref{prob_parab-1a} and \ref{prob_nonlin},   we need to work with operator pencils generated by  the corresponding elliptic problems in the polyhedral type domain ${D}\subset \real^3$.

We briefly recall the basic facts needed in the sequel. For further information  on this subject we refer to  \cite{KMR01} and  \cite[Sect.~2.3, 3.2., 4.1]{MR10}. On a domain 
 $D\subset\real^3$  of polyhedral type according to Definition \ref{standard} we consider the probelm 
\begin{equation} 
\left\{\begin{array}{rl}\label{prob-ell-00}
Lu\ =\ f \, & \  \text{in} \quad D, \\
\frac{\partial^{k-1}u}{\partial \nu^{k-1}}\Big|_{\partial D}\ =\ 0, & \   k=1,\ldots, m.   
\end{array} \right\} 
\end{equation}

The singular set $S$ of $D$ then is given by the boundary points   $M_1\cup \ldots \cup M_l\cup \{x^{(1)}, \ldots, x^{(l')}\}$. We do not exclude the cases $l=0$ (corner domain) and $l'=0$ (edge domain). In the last case, the set $S$ consists only of smooth non-intersecting  edges. Figure  \ref{corner-edge-dom}  gives examples of polyhedral domains without edges or corners, respectively. \index{domain! corner domain}\index{domain! edge domain}\\  

\begin{figure}[H]
\begin{minipage}{\textwidth}
\begin{center}
\includegraphics[width=8cm]{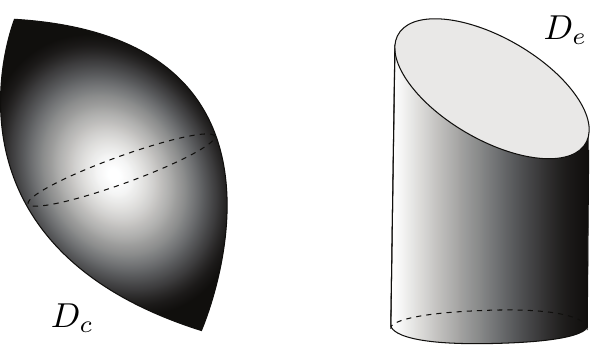}
\caption[Corner and edge domain]{Corner domain $D_c$ ($l=0$) and edge domain $D_e$ ($l'=0$)}
\label{corner-edge-dom}
\end{center}
\end{minipage}\\
\end{figure}

The elliptic boundary value problem \eqref{prob-ell-00} on $D$ generates two types of operator pencils for the edges $M_k$ and for the vertices $x^{(i)}$ of the domain, respectively.

\begin{figure}[H]
\begin{minipage}{0.5\textwidth}
{\bf 1) Operator pencil $A_{\xi}(\lambda)$ for edge points:} \\
The pencils $A_{\xi}(\lambda)$ for  edge points $\xi\in M_k$ are defined as follows: According to Definition \ref{standard} there exists a neighborhood $U_{\xi}$ of $\xi$ and a diffeomorphism $\kappa_{\xi}$  mapping $D\cap U_{\xi}$ onto $\mathcal{D}_{\xi} \cap B_1(0)$, where $\mathcal{D}_{\xi}$ is a  dihedron. \\
Let $\Gamma_{k_{\pm}}$ be the faces adjacent to $M_k$. Then by $\mathcal{D}_{\xi}$ we denote the dihedron which is bounded by the half-planes $\mathring{\Gamma}_{k_{\pm}}$ tangent to $\Gamma_{k_{\pm}}$ at $\xi$ and the edge $M_{\xi}=\mathring{\Gamma}_{k_{+}}\cap \mathring{\Gamma}_{k_{-}}$. Furthermore, let $r,\varphi$ be polar coordinates in the plane perpendicular to $M_{\xi}$ such that 
\[
\mathring{\Gamma}_{k_{\pm}}=\left\{x\in \real^3: \ r>0, \ \varphi=\pm \frac{\theta_{\xi}}{2}\right\}.
\]
\end{minipage}\hfill \begin{minipage}{0.4\textwidth}
\includegraphics[width=6cm]{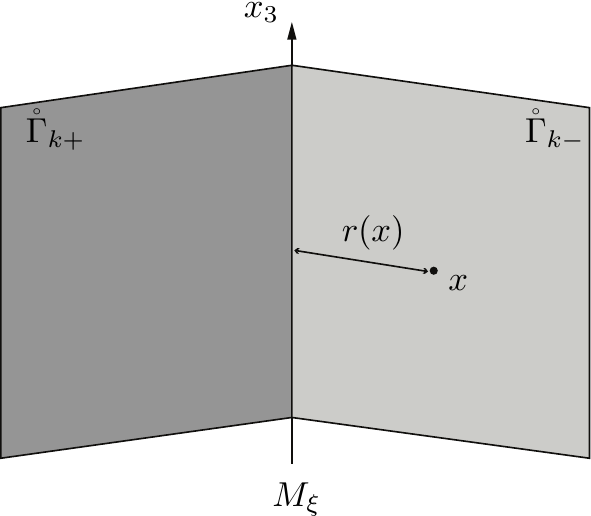}
\caption[Dihedron $\mathcal{D}_{\xi}$]{Dihedron $\mathcal{D}_{\xi}$}
\end{minipage}\\
\end{figure}
We define the {\em operator pencil} $A_{\xi}(\lambda)$ as  follows:   \index{operator pencil! $A_{\xi}(\lambda)$}
\beq\label{op-pencil-1}
A_{\xi}(\lambda)U(\varphi)=r^{2m-\lambda}L_{0}(0,D_x)u, 
\eeq
where $u(x)=r^{\lambda}U(\varphi)$,  $\lambda\in \mathbb{C}$,  {$U$ is a function on $I_{\xi}:=\left(\frac{-\theta_{\xi}}{2}, \frac{\theta_{\xi}}{2}\right)$,}  and 
\[
L_{0}(\xi,D_x)=\sum_{|\alpha|=|\beta|=m}D^{\alpha}_x(a_{\alpha\beta}(\xi)D_x^{\beta}) 
\] 
denotes the main part of the differential operator $L(x,D_x)$ with coefficients frozen at $\xi$. 
This way we obtain in \eqref{op-pencil-1} a boundary value problem for the function $U$ on the 1-dimensional subdomain $I_{\xi}$ with the complex parameter $\lambda$.  Obviously,  $A_{\xi}(\lambda)$  is a polynomial of degree $2m$ in $\lambda$.

The operator $A_{\xi}(\lambda)$ realizes a continuous mapping 
\[
H^{2m}(I_{\xi})\rightarrow L_2(I_{\xi}),
\]
for every $\lambda\in \mathbb{C}$. 
Furthermore, $A_{\xi}(\lambda)$ is an isomorphism for all $\lambda\in \mathbb{C}$ with the possible exception of a denumerable set of isolated points, the {\em spectrum of  $A_{\xi}(\lambda)$}, \index{operator pencil! spectrum} which  consists of its eigenvalues with finite algebraic multiplicities: Here a  complex number $\lambda_0$ is called an {\em eigenvalue of the pencil $A_{\xi}(\lambda)$}\index{operator pencil! eigenvalue} if there exists  a nonzero function $U\in H^{2m}(I_{\xi})$ such that $A_{\xi}(\lambda_0)U=0$.
 It is known that the {\em 'energy line'} $\mathrm{Re}\lambda=m-1$  does not contain eigenvalues of the pencil $A_{\xi}(\lambda)$. We denote by {$\delta_{\pm}^{(\xi)}$} the largest positive real numbers such that the strip 
\beq\label{delta_k_op-1}
m-1-\delta_{-}^{(\xi)}<\mathrm{Re}\lambda<m-1+\delta_{+}^{(\xi)}
\eeq
is free of eigenvalues of the pencil $A_{\xi}(\lambda)$. Furthermore, we put 
\beq\label{delta_k_op}
{\delta_{\pm}^{(k)}}=\inf_{\xi\in M_k}{\delta_{\pm}^{(\xi)}}, \qquad k=1,\ldots, l. 
\eeq 

For example, concerning the Dirichlet problem for the Poisson equation 
on a domain $D\subset \real^3$  of polyhedral type, the eigenvalues of the pencil $A_{\xi}(\lambda)$  are given by 
\[
\lambda_k=k\pi/\theta_{\xi}, \qquad k=\pm1, \pm2, \ldots, 
\]
where   $\theta_{\xi}$ is the inner angle at the edge point $\xi$, cf. \cite[Ex. 2.5.2]{co-habil}. 
Therefore, the first positiv eigenvalue is $\lambda_1=\frac{\pi}{\theta_{\xi}}$ and we obtain $\delta_{\pm}=\frac{\pi}{\theta_{\xi}}$, cf. \cite[Ex. 2.5.1]{co-habil}.  \\

{\bf 2) Operator pencil $\mathfrak{A}_i(\lambda)$ for corner points:} \\ Let $x^{(i)}$ be a vertex of $D$. According to Definition \ref{standard} there exists a neighborhood $U_i$ of $x^{(i)}$ and a diffeomorphism $\kappa_i$  mapping $D\cap U_i$ onto $K_i\cap B_1(0)$ , where 
\[
K_i=\{x\in \real^3: \ x/|x|\in \Omega_i\}
\]
is a polyhedral cone with edges and vertex at the origin. W.l.o.g. we may assume that the Jacobian matrix $\kappa_i'(x)$ is equal to the identity matrix at the point $x^{(i)}$. We introduce spherical coordinates  $\rho=|x|$, $\omega=\frac{x}{|x|}$ in $K_i$ and define the operator pencil \index{operator pencil! $\mathfrak{A}_{i}(\lambda)$}
\begin{equation}\label{op-pencil}
\mathfrak{A}_i(\lambda)U(\omega)=\rho^{2m-\lambda}L_{0}(x^{(i)},D_x)u,
\end{equation}
where $u(x)=\rho^{\lambda}U(\omega)$ and $U\in \mathring{H}^{m}(\Omega_i)$ is a function on $\Omega_i$. An {\em eigenvalue of  $\mathfrak{A}_i(\lambda)$}\index{operator pencil! eigenvalue} is a complex number $\lambda_0$ such that 
$\mathfrak{A}_i(\lambda_0)U=0$ for some nonzero function $U\in\mathring{H}^{m}(\Omega_i)$. 
The operator $\mathfrak{A}_i(\lambda)$ realizes a continuous mapping 
\[
\mathring{H}^{m}(\Omega_i)\rightarrow H^{-m}(\Omega_i).
\]
Furthermore, it is  known that  $\mathfrak{A}_i(\lambda)$ is an isomorphism for all $\lambda\in \mathbb{C}$ with the possible exception of a denumerable set of isolated points. The mentioned enumerable set consists of eigenvalues with finite algebraic multiplicities. 
\begin{figure}[H]
\begin{minipage}{0.4\textwidth}
Moreover, the eigenvalues of $\mathfrak{A}_i(\lambda)$ are situated, except for finitely many, outside a double sector $|\mathrm{Re}\lambda|<\varepsilon |\mathrm{Im}\lambda|$ containing the imaginary axis, cf. \cite[Thm. 10.1.1]{KMR01}. In Figure \ref{eigenvalues-pencil} the situation is illustrated: Outside the yellow area there are only finitely many eigenvalues of the operator pencil $\mathfrak{A}_i(\lambda)$.   \\
Dealing with regularity properties of solutions, we look for the widest strip in the $\lambda$-plane, free of eigenvalues and containing the {\em 'energy line'}  $ \mathrm{Re}\lambda=m-3/2,$ 
cf. Assumption \ref{assumptions}. From what was outlined above, information on the width of this strip is obtained from lower estimates for real parts of the eigenvalues situated over the energy line.   \\
\end{minipage}\hfill \begin{minipage}{0.5\textwidth}
\includegraphics[width=15cm]{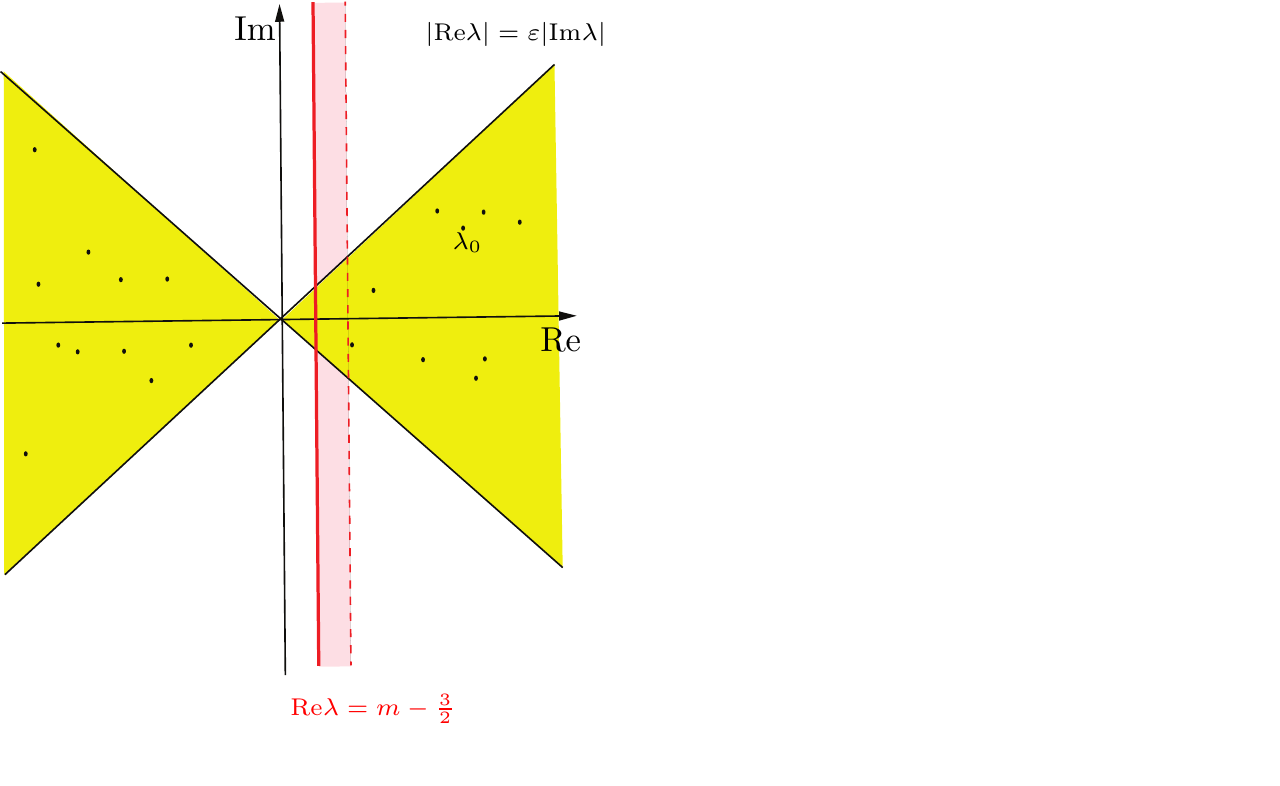}
\caption[Eigenvalues of pencil $\mathfrak{A}_i(\lambda)$]{Eigenvalues of operator pencil $\mathfrak{A}_i(\lambda)$}
\label{eigenvalues-pencil}
\end{minipage}
\end{figure}

\begin{rem}[{\bf Operator pencils for parabolic problems}]  
Since  we study parabolic PDEs, where the differential operator $L(t,x,D_x)$ additionally depends on the time $t$, we have to work with operator pencils $A_{\xi}(\lambda,t)$ and $\mathfrak{A}_{i}(\lambda,t)$ in this context. The philosophy  is to fix $t\in [0,T]$ and define the pencils as above: We replace  \eqref{op-pencil-1} by 
\[
A_{\xi}(\lambda,t)U(\varphi)=r^{2m-\lambda}L_{0}(t,0,D_x)u, 
\]
and work with $\delta^{(\xi)}_{\pm}(t)$ and $\delta_{\pm}^{(k)}(t)=\inf_{\xi\in M_k}{\delta_{\pm}^{(\xi)}}(t)$ in \eqref{delta_k_op-1} and \eqref{delta_k_op}, respectively. Moreover, we put  
\beq\label{delta_k_op_t}
{\delta_{\pm}^{(k)}}=\inf_{t\in [0,T]}{\delta_{\pm}^{(k)}}(t), \qquad k=1,\ldots, l. 
\eeq 
Similar for $\mathfrak{A}_{i}(\lambda,t)$, where now \eqref{op-pencil} is replaced by 
\begin{equation}
\mathfrak{A}_i(\lambda,t)U(\omega)=\rho^{2m-\lambda}L_{0}(t,x^{(i)},D_x)u. 
\end{equation}
\end{rem}

\section{Regularity results in Sobolev and Kondratiev spaces}
\label{sect-reg-sob-kon}

This section presents regularity results for Problems \ref{prob_parab-1a} and \ref{prob_nonlin} in Sobolev and Kondratiev spaces. They will form the basis for obtaining regularity results in Besov  spaces later on via suitable embeddings. 
The results in Sobolev and Kondratiev spaces for Problems \ref{prob_parab-1a} and \ref{prob_nonlin} on domains of polyhedral type  $D\subset \real^d$ are  essentially new and not published elsewhere so far: In \cite{DS19} we restricted our investigations to polyhedral cones $K\subset \real^3$ relying on the results from \cite{LL15}. \\ 
However, the extension of the regularity results for Problem \ref{prob_parab-1a} to polyhedral type domains follows from very similar arguments as in \cite{DS19}, which is why we merely state the results in Sections \ref{subsect-sob-reg} and \ref{subsect-kondr-prob1} and give references for the proofs wherever necessary. In contrast to this the regularity results for the nonlinear Problem  \ref{prob_nonlin} require some careful adaptations and are carried out in detail in Section \ref{subsect-reg-nonlin}.  \\

\subsection{Regularity results in Sobolev  spaces for Problem I}
\label{subsect-sob-reg}

In this subsection we are concerned with the Sobolev  regularity of the weak solution of Problem \ref{prob_parab-1a}. 
We start with the following lemma, whose proof is similar to \cite[Lem.~4.1]{AH08}.  

\begin{lemma}[{\bf Continuity of bilinear form}]\label{lem-cont-F}
Assume that for each $t\in [0,T]$, $F(t,\cdot,\cdot): \mathring{H}^m(D)\times \mathring{H}^m(D)\rightarrow \real$ is a bilinear map satisfying 
\begin{equation}\label{cont-F}
|F(t,u,v)|\leq C \|u| \mathring{H}^m(D)\| \|v| \mathring{H}^m(D)\|
\end{equation}
for all $t\in [0,T]$ and all $u,v\in \mathring{H}^m(D)$, where $C$ is a constant independent of $u,v$, and $t$. Assume further that $F(\cdot, u,v)$ is measurable on $[0,T]$ for each pair $u,v\in \mathring{H}^m(D)$. Assume that $u\in  {H}^{m,1*}(D_T)$ satisfies $u(0)\equiv 0$ and 
\begin{equation}\label{cond-bilinear}
(\partial_t u(t),v)+B(t,u(t),v)=\int_0^t F(\tau,u(\tau),v)\ud \tau
\end{equation} 
for a.e. $t\in [0,T]$ and all $v\in \mathring{H}^m(D)$. Then $u\equiv 0$ on $[0,T]\times D$. 
\end{lemma}

Using the spectral method  the following regularity result now follows. 

\begin{theorem}[{\bf Sobolev regularity without time derivatives}]\label{Sob-reg-2} \hfill \\
Let $f\in L_2([0,T],H^{-m}(D))$. 
Then Problem \ref{prob_parab-1a} has a unique weak solution $u$ in the space ${H}^{m,1*}(D_T)$ and the following estimate holds 
\begin{equation}\label{Sob-reg-1}
\|u|H^{m,1*}(D_T)\|\leq C \|f|L_2([0,T],H^{-m}(D))\|,
\end{equation}
where $C$ is a constant independent of $f$ and $u$. 
\end{theorem}

 This proof follows  \cite[Lem.~4.2]{AH08}, which in turn is based on \cite[Sect.~7.1.2]{Eva10}. 

By an application of Theorem \ref{Sob-reg-2} and induction we obtain the following regularity result. The proof is similar to \cite[Thm.~2]{ALL16}.  

\begin{theorem}[{\bf Sobolev regularity with time derivatives}]\label{Sob-reg-3}
Let $l\in \nat_0$ and assume that the right-hand side $f$ of Problem \ref{prob_parab-1a} satisfies 
$$f\in H^l([0,T], H^{-m}(D))\qquad \text{and} \qquad 
\partial_{t^k}f(x,0)=0  \quad  \text{ for } \quad  k=0,\ldots, l-1.$$ 
Then the weak solution $u$ in the space ${H}^{m,1*}(D_T)$ of Problem \ref{prob_parab-1a} in fact belongs to ${H}^{m,l+1*}(D_T)$, i.e., 
has derivatives with respect to $t$ up to order $l$ satisfying 
\[
\partial_{t^k}u\in {H}^{m,1*}(D_T)\quad \text{for}\quad k=0,\ldots, l,
\]
and 
\[
\sum_{k=0}^l\|\partial_{t^k}u|H^{m,1*}(D_T)\| \leq C\sum_{k=0}^l\|\partial_{t^k} f|L_2([0,T],H^{-m}(D))\|, 
\]
where $C$ is a constant independent of $u$ and $f$. 
\end{theorem}

\begin{rem}
Note that the regularity results for the solution $u$ in  \cite[Thm.~2.1., Lem.~3.1]{LL15} are slightly stronger than the ones obtained in Theorem \ref{Sob-reg-3} above (with the cost of also assuming more regularity on the right-hand side $f$).  
By using similar arguments as in  \cite[Lem.~4.3]{AH08} we are probably able to also show  in our context that Theorem \ref{Sob-reg-2} can be strengthened in the sense  that if  $f\in L_2([0,T], L_2(D))$ then the weak solution $u$ of Problem \ref{prob_parab-1a} belongs in fact  to $L_2([0,T], \mathring{H}^m)\cap H^1([0,T], L_2(D))$. A corresponding generalization of Theorem \ref{Sob-reg-3} should also be possible in the spirit of \cite[Thm.~3.1]{AH08}. However, for our purposes the above results on the Sobolev regularity are sufficient, so these investigations are postponed for the time being. 
\end{rem}

\subsection{Regularity results in Kondratiev spaces for Problem I}
\label{subsect-kondr-prob1}

Concerning weighted Sobolev regularity of  Problem  \ref{prob_parab-1a} first fundamental results on polyhedral cones $K\subset \real^3$ can be found in 
{ \cite[Thms.~3.3,~3.4]{LL15}. In \cite{DS19} we extended and generalized these results, which we now wish to transfer to our setting of polyhedral type domains $D\subset \real^3$.  

For our regularity assertions  we rely on known results for elliptic equations. Therefore, we consider  first 
the following Dirichlet problem for elliptic equations 
\begin{equation}
\left.\begin{cases}Lu=F& \text{on}\quad  {D},\\
\frac{\partial^k u}{\partial \nu^k}\big|_{\Gamma_j}=0, & k=1,\ldots, m, \ j=1,\ldots, n,  
\end{cases}\right\}
\qquad \label{ellipt-pde}\end{equation}
where $D\subset \real^3$ is a domain of polyhedral type according to Definition \ref{standard} with faces $\Gamma_j$. Moreover, we assume that  
\[
L(x,D_x)=\sum_{|\alpha|\leq 2m}A_{\alpha}(x)D^{\alpha}_x
\]
is a uniformly elliptic differential operator of order $2m$ with smooth coefficients $A_{\alpha}$. 
We  need the following technical assumptions in order to state  the Kondratiev regularity of \eqref{ellipt-pde}. 

\begin{assumption}[{\bf Assumptions on operator pencils}]\label{assumptions} 
Consider the operator pencils $\mathfrak{A}_i(\lambda,t)$,   $i=1,\ldots, l'$ for the vertices  and $A_{\xi}(\lambda,t)$ with $\xi\in M_k$, $k=1,\ldots, l$ for the edges of the polyhedral type domain $D\subset \real^3$ introduced in Section \ref{subsect-op-pen}.  For the elliptic problem \eqref{ellipt-pde} we may drop $t$ from the notation of the pencils, otherwise (for our parabolic problems) we assume $t\in[0,T]$ is fixed. \\
Let $\calk^{\gamma}_{p,b}(D)$ and $\calk^{\gamma'}_{p,b'}(D)$ be two Kondratiev spaces, where the singularity set $S$ of $D$ is given by $S=M_1\cup \ldots\cup M_l\cup \{x^{(1)},\ldots, x^{(l')}\}$ and weight parameters $b,b'\in \real$.   Then we  assume that the closed strip between the lines 
\beq\label{op-pen-ass1}
\mathrm{Re}\lambda=b+2m-\frac 32\qquad \text{and}\qquad \mathrm{Re}\lambda=b'+2m-\frac 32
\eeq 
does not contain eigenvalues of    $\mathfrak{A}_i(\lambda,t)$. Moreover, $b$ and $b'$ satisfy
\begin{equation}\label{restr-1-a}
-\delta_-^{(k)}<b+m<\delta_{+}^{(k)}, \qquad -\delta_-^{(k)}<b'+m<\delta_{+}^{(k)}, \quad k=1,\ldots, l, 
\end{equation}
where $\delta_{\pm}^{(k)}$ are defined in \eqref{delta_k_op} (replaced by \eqref{delta_k_op_t} for parabolic problems). 
\end{assumption}

\begin{rem}
 If $l'=0$ we have an edge domain without vertices, cf.   Figure \ref{corner-edge-dom}. In this case condition  \eqref{op-pen-ass1} is empty. Moreover, if $l=0$, we have a corner  domain without edges, in which case condition \eqref{restr-1-a} is empty. 
 For further remarks and explanations  concerning Assumption \ref{assumptions} we refer to \cite[Rem.~3.3]{DS19}.
\end{rem}

\medskip

The following lemma on the regularity of solutions to elliptic boundary value problems in domains of polyhedral type is taken from \cite[Cor.~4.1.10, Thm.~4.1.11]{MR10}. We rewrite it for  our scale of Kondratiev spaces.

\begin{lemma}[{\bf Kondratiev regularity for elliptic PDEs}] \label{mazja_ross} 
Let $D\subset \real^3$ be a domain of polyhedral type. Moreover, let $u\in \mathcal{K}^{\gamma}_{2,a+2m}(D)$ be a solution of \eqref{ellipt-pde}, where 
\[
F\in \mathcal{K}^{\gamma-2m}_{2,a}(D)\cap \mathcal{K}^{\gamma'-2m}_{2,a'}(D), \qquad \gamma\geq m, \quad \gamma'\geq m. 
\] 
{Suppose that $\mathcal{K}^{\gamma}_{2,a}(D)$ and $\mathcal{K}^{\gamma'}_{2,a'}(D)$ satisfy Assumption \ref{assumptions}.} 
 Then $u\in \mathcal{K}^{\gamma'}_{2,a'+2m}(D)$ and 
 \[
 \|u|\mathcal{K}^{\gamma'}_{2,a'+2m}(D)\|\leq C\|F|\mathcal{K}^{\gamma'-2m}_{2,a'}(D)\|, 
 \] 
where $C$ is a constant independent of $u$ and $F$. 
\end{lemma}

\begin{rem}\label{rem_mazja_ross_ell}
In particular, if in Theorem \ref{Sob-reg-3} we use the stronger assumption  $\partial_{t^k}f(t) \in L_2(D)$ instead of $\partial_{t^k}f(t)\in H^{-m}(D)$   for $k=0,\ldots, l$, then it follows that 
\begin{equation}\label{ell_1}
\partial_{t^k}f(t) \in L_2(D)=\mathcal{K}^0_{2,0}(D) \hookrightarrow \mathcal{K}^{-m}_{2,-m}(D), 
\end{equation}
where the latter embedding follows from the corresponding duality assertion, i.e., we have  
$ 
\mathcal{K}^m_{2,m}(D) \hookrightarrow \mathcal{K}^{0}_{2,0}(D)
$  
since $m\geq 0$. In this case the solution $u$ of  Problem \ref{prob_parab-1a} satisfies 
\begin{equation}\label{ell_2}
\partial_{t^k}u(t)\in \mathring{H}^m(D)\hookrightarrow \mathring{\mathcal{K}}^m_{2,m}(D)\hookrightarrow \mathcal{K}^0_{2,a}(D), \qquad {a\leq m}, 
\end{equation}
where the first embedding is taken from \cite[Lem.~3.1.6]{MR10} and the second embedding for Kondratiev spaces holds whenever 
$m\geq a$. 
{We additionally require in our considerations that 
$
\partial_{t^k}u(t)\in \mathcal{K}^0_{2,a}(D)\hookrightarrow \mathcal{K}^{-m}_{2,-m}(D)$
which holds for $a\geq -m$. 
}
From \eqref{ell_1} and \eqref{ell_2} we see that it is possible to take $\gamma=m$ and $a=-m$ in  Lemma \ref{mazja_ross}, {
i.e., if  $f(t)\in \calk^{-m}_{2,-m}(D)$ then $u(t)\in \calk^m_{2,m}(D)$. {Note that all our arguments with $u(t)$ and $f(t)$, respectively, hold for a.e. $t\in [0,T]$. However, since lower order time derivatives are continuous w.r.t. suitable spaces (but not necessarily the highest one, cf. the proof of Thm. \ref{Hoelder-Besov-reg}), we will suppress this distinction in the sequel.}
}
\end{rem}

Using similar arguments as in \cite[Thm.~3.3]{LL15} we are now able to show the following regularity result in Kondratiev spaces. The proof follows along the same lines as  \cite[Thm.~3.6]{DS19}. \\ 

\begin{theorem}[{\bf Kondratiev regularity A}]\label{thm-weighted-sob-reg}
Let $D\subset \real^3$  be a domain of polyhedral type. Let $\gamma\in \nat $ with  $\gamma\geq 2m$ and put $\gamma_m:=\left[ \frac{\gamma-1}{2m}\right]$. Furthermore, let  $a\in \real$ with   ${a\in [-m,m]}$.  Assume that the right-hand side $f$ of Problem \ref{prob_parab-1a} satisfies 
\begin{itemize}
\item[(i)] $\partial_{t^k} f\in L_2(D_T)\cap L_2([0,T],\mathcal{K}^{2m(\gamma_m-k)}_{2,a+2m(\gamma_m-k)}(D))$, \ $k=0,\ldots, \gamma_m$; \quad 
$\partial_{t^{\gamma_m+1}} f\in L_2(D_T)$. 
\item[(ii)] $\partial_{t^k} f(x,0)=0$, \quad  $k=0,1,\ldots, {\gamma_m}.$
\end{itemize}
{Furthermore, let  Assumption \ref{assumptions}  hold for weight parameters $b=a+2m(\gamma_m-i)$, where $i=0,\ldots, \gamma_m$, and  $b'=-m$.}
Then for the weak solution $u\in {{H}}^{m,\gamma_m+2\ast}(D_T)$ of Problem \ref{prob_parab-1a} we have 
$$\partial_{t^{l+1}} u\in L_2([0,T],\mathcal{K}^{2m(\gamma_m-l)}_{2,a+2m(\gamma_m-l)}(D))$$ for $l=-1,0,\ldots, \gamma_m$. In particular, for the derivatives $\partial_{t^{l+1}} u$ up to order $\gamma_m+1$ we have the a priori estimate 
\begin{align}
\sum_{l=-1}^{\gamma_m}& \|\partial_{t^{l+1}} u|{L_2([0,T],\mathcal{K}^{2m(\gamma_m-l)}_{2,a+2m(\gamma_m-l)}(D))}\|\notag\\
&\lesssim  \sum_{k=0}^{\gamma_m}\|\partial_{t^k} f|{L_2([0,T], \mathcal{K}^{2m(\gamma_m-k)}_{2,a+2m(\gamma_m-k)}(D))}\|+\sum_{k=0}^{\gamma_m+1}\|\partial_{t^k} f|{L_2(D_T)}\|,\label{weighted-sobolev-est}
\end{align}
where the  constant is  independent of $u$ and $f$. 
\end{theorem}

\remark{\label{rem-restr}
 The existence of the solution $u\in {H}^{m,{\gamma_m+2}\ast}(D_T)$  follows from  Theorem \ref{Sob-reg-3} {using  $l=\gamma_m+1$}.
}

The regularity results obtained in Theorem \ref{thm-weighted-sob-reg} only hold under certain restrictions on the parameter $a$ we are able to choose. In particular, we cannot choose $\gamma_m>0$ if we have a non-convex polyhedral type domains $D$, since there is no suitable $a$ satisfying all of our requirements in this case. In order to treat non-convex domains as well, we impose stronger assumptions on the right--hand side $f$, requiring that it is  arbitrarily smooth w.r.t. the time. This additional assumption allows for a larger range of  $a$. However, as a drawback, these results are hard to apply to nonlinear equations since the right-hand sides are not taken from a Banach  or quasi-Banach space. The proof of the following theorem is similar to  \cite[Thm.~3.9]{DS19} adapted to our setting.\\

\begin{theorem}[{\bf Kondratiev regularity B}]\label{thm-weighted-sob-reg-2}
Let $D\subset \real^3$ be a domain of polyhedral type and  $\eta\in \nat $ with  $\eta\geq 2m$. Moreover, let  $l\in \nat_0$ and $a\in \real$ with   $a\in [-m,m]$.  Assume that the right-hand side $f$ of Problem \ref{prob_parab-1a} satisfies 
\begin{itemize}
\item[(i)] $ f\in \bigcap_{l=0}^{\infty}H^l([0,T], L_2(D)\cap \mathcal{K}^{\eta-2m}_{2,a}(D))$. 
\item[(ii)] $\partial_{t^l} f(x,0)=0$, \quad  $l\in \nat_0.$
\end{itemize}
Furthermore, let  Assumption \ref{assumptions}  hold for weight parameters $b=a$ and  $b'=-m$. 
Then for the weak solution $u\in  \bigcap_{l=0}^{\infty}{{H}}^{m,l+1\ast}(D_T)$ of Problem \ref{prob_parab-1a} we have 
$$\partial_{t^{l}} u\in L_2([0,T],\mathcal{K}^{\eta}_{2,a+2m}(D))\qquad \text{for all}\quad l\in \nat_0. $$
In particular, for the derivative $\partial_{t^l} u$  we have the a priori estimate 
\begin{align}
\sum_{k=0}^{l}& \|\partial_{t^{k}} u|{L_2([0,T],\mathcal{K}^{\eta}_{2,a+2m}(D))}\|\notag\\
&
\lesssim  \sum_{k=0}^{l+(\eta-2m)}\|\partial_{t^k} f|{L_2([0,T], \mathcal{K}^{\eta-2m}_{2,a}(D))}\|+\sum_{k=0}^{l+1+(\eta-2m)}\|\partial_{t^k} f|{L_2(D_T)}\|, \notag 
\end{align}
where the  constant is  independent of $u$ and $f$. 
\end{theorem}

\remark{In Theorem \ref{thm-weighted-sob-reg-2} compared to Theorem \ref{thm-weighted-sob-reg} we only require the parameter $a$ to satisfy $a\in [-m,m]$ and $-\delta_{-}^{(k)}<a+m<\delta_{+}^{(k)}$ independent of the regularity parameter {$\eta$} which can be arbitrarily high. In particular, for the heat equation on 
a  domain of polyhedral type $D$ (which for simplicity we assume to be  a polyhedron with straight edges and faces where   $\theta_k$ denotes the angle at the edge $M_k$), we have $\delta_{\pm}^{(k)}=\frac{\pi}{\theta_k}$, which leads to the restriction 
$ 
-1\leq  a<\min\left(1, \frac{\pi}{\theta_k}-1\right).  
$ 
Therefore,  even in the extremal case when $\theta_k=2\pi$ we can still take $-1\leq a<-\frac 12$ (resulting in $u\in L_2([0,T], \calk^{\eta}_{a+2}(D))$ being locally integrable since $a+2>0$). Then choosing  {$\eta$} arbitrary high,  we  also cover non-convex polyhedral type domains with our results from Theorem \ref{thm-weighted-sob-reg-2}. 
}

\subsection{Regularity results in Sobolev and Kondratiev spaces for Problem II}
\label{subsect-reg-nonlin}

In this subsection we show that the regularity estimates in {Kondratiev and Sobolev spaces  as stated in Theorems \ref{thm-weighted-sob-reg} and  \ref{Sob-reg-3}}, respectively,  carry over to Problem \ref{prob_nonlin}, provided that $\varepsilon$ is sufficiently small. In order to do this we interpret Problem \ref{prob_nonlin} as a fixed point problem in the following way. 

 Let $\widetilde{\mathcal{D}}$ and $S$ be Banach spaces ({$\widetilde{\mathcal{D}}$ and $S$ will be specified in the theorem below}) and let $\tilde{L}^{-1}:\widetilde{\mathcal{D}}\rightarrow S$ be the linear operator defined  via
\begin{equation} \label{tildeL}
\tilde{L}u:=\frac{\partial}{\partial t}u+(-1)^mLu. 
 \end{equation}
 Problem \ref{prob_nonlin}  is equivalent to 
\[
\tilde{L}u=f-\varepsilon u^{M}=:Nu,
\] 
where $N:S\rightarrow \widetilde{\mathcal{D}}$ is a nonlinear operator. If we can show that $N$ maps $S$ into $\widetilde{\mathcal{D}}$, then a solution of Problem \ref{prob_nonlin} is a fixed point of the problem 
\[
(\tilde{L}^{-1}\circ N)u=u.
\]
Our aim is  to apply Banach's fixed point theorem, \index{Fixed point theorem! Banach} which will also guarantee uniqueness of the solution, if we can show that  $T:=(\tilde{L}^{-1}\circ N): S_0\rightarrow S_0$ is a contraction mapping, i.e., there exists some $q\in [0,1)$ such that 
\[
\|T(x)-T(y)|S\|\leq q\|x-y|S\| \quad \text{for all}\quad x,y\in S_0, 
\]
where the corresponding subset $S_0\subset S$ is a small {closed} ball with center $\tilde{L}^{-1}f$  (the solution of the corresponding linear problem) and suitably chosen radius $R>0$. \\

Our main result is stated in the theorem below. 

\begin{theorem}[{\bf Nonlinear Sobolev and Kondratiev regularity}]\label{nonlin-B-reg1}
Let $\tilde{L}$ and $N$ be as described above. Assume the assumptions of Theorem \ref{thm-weighted-sob-reg} are satisfied and, additionally, we have  {$\gamma_m\geq 1$}, $m\geq  2$,  and {$a\geq -\frac 12$}.  
Let 
\[
\D_1:=\bigcap_{k=0}^{\gamma_m}H^{k}([0,T],\mathcal{K}^{2m(\gamma_m-k)}_{2,a+2m(\gamma_m-k)}(D)), \quad 
\D_2:=H^{\gamma_m+1}([0,T],L_2(D))
\]
and consider the data space 
\begin{align*}
\widetilde{\mathcal{D}}&:=\{f\in \D_1 
\cap \D_2: \  
\partial_{t^k} f(0,\cdot)=0, \quad k=0,\ldots, \gamma_m\}.   
\end{align*}
Moreover, let 
\begin{align*}
S_1&:= \bigcap_{k=0}^{\gamma_m+1}H^{k}([0,T],\mathcal{K}^{2m(\gamma_m-(k-1))}_{2,a+2m(\gamma_m-(k-1))}(D)), \qquad
S_2:={H}^{m,\gamma_m+2\ast}(D_T), 
\end{align*}
and  consider the solution space 
$S:=S_1\cap S_2$. Suppose that $f\in \widetilde{\mathcal{D}}$
and put  $\eta:=\|f|\widetilde{\mathcal{D}}\|$ and $r_0>1$. Moreover, we choose $\varepsilon >0$ so small that 
\[
{
\eta^{2(M-1)} \|\tilde{L}^{-1}\|^{2M-1}\leq \frac{1}{{c}\varepsilon M}(r_0-1)\left(\frac{1}{r_0}\right)^{2M-1}, \qquad \text{if}\quad  r_0\|\tilde{L}^{-1}\|\eta>1,
}
\]
and 
\[\|\tilde{L}^{-1}\|<\frac{r_0-1}{r_0}\left(\frac{1}{{c}\varepsilon M}\right), \qquad \text{if}\quad  r_0\|\tilde{L}^{-1}\|\eta<1,\]
{where $c>0$ denotes the  constant in  \eqref{est-ab} resulting from our estimates below. }
Then there exists a unique  solution $u\in S_0\subset S$ of Problem \ref{prob_nonlin}, where $S_0$ denotes a small ball  around $\tilde{L}^{-1}f$ (the solution of the corresponding linear problem) with radius $R=(r_0-1)\eta \|\tilde{L}^{-1}\|$. 
\end{theorem}

\begin{proof} 
Let $u$ be the solution of the linear problem $\tilde{L}u=f$.  From Theorems \ref{thm-weighted-sob-reg} and  \ref{Sob-reg-3} we know   that 
$\tilde{L}^{-1}: \widetilde{\mathcal{D}}\rightarrow S   
$ 
is a bounded operator.  
If  $u^M\in \widetilde{\mathcal{D}}$ (this will immediatelly follow from our calculations in Step 1  as explained in Step 2 below), the nonlinear part $N$ satisfies the desired mapping properties, i.e., $Nu=f-\varepsilon u^M\in \widetilde{\mathcal{D}}$  
  and 
we can apply Theorem  \ref{thm-weighted-sob-reg}  now with right-hand side $Nu$. \\
{\em Step 1: }  Since 
\[
 (\tilde{L}^{-1}\circ N)(v)-(\tilde{L}^{-1}\circ N)(u)= \tilde{L}^{-1}(f-\varepsilon v^{M})-\tilde{L}^{-1}(f-\varepsilon u^{M}) =\varepsilon  \tilde{L}^{-1}(u^{M}-v^{M})
\]
one sees that $\tilde{L}^{-1}\circ N$ is a contraction if, and only, if 
\begin{equation}\label{est-0}
\varepsilon \| \tilde{L}^{-1}(u^{M}-v^{M})|S\| \leq q\|u-v|S\|\quad \text{ for some }\quad q<1, 
\end{equation}
where $u,v\in S_0$ (meaning $u,v\in   B_R(\tilde{L}^{-1}f)$ in $S$). 
We analyze the {resulting  condition} with the help of the formula  $ u^M-v^M=(u-v)\sum_{j=0}^{M-1} u^jv^{M-1-j}$. This together with Theorem \ref{thm-weighted-sob-reg} gives 
\begin{align}
\|\tilde{L}^{-1}(u^M-v^M)|S\|
&\leq \|\tilde{L}^{-1}\| \|u^M-v^M|\widetilde{\mathcal{D}}|\notag\\
&= \|\tilde{L}^{-1}\| \left\|u^M-v^M|
\D_1\cap \D_2 \right\|\notag\\
&= \|\tilde{L}^{-1}\| \left(\|u^M-v^M|\D_1\|+\|u^M-v^M|\D_2\|\right) \notag \\
&=  \|\tilde{L}^{-1}\| \left(\left\|(u-v)\sum_{j=0}^{M-1} u^jv^{M-1-j}|
\D_1
\right\|+ \left\|(u-v)\sum_{j=0}^{M-1} u^jv^{M-1-j}|
\D_2
\right\| \right)\notag\\
&\lesssim   \|\tilde{L}^{-1}\| \Bigg( \sum_{k=0}^{\gamma_m}\left\|\partial_{t^k}\left[(u-v)\sum_{j=0}^{M-1} u^jv^{M-1-j}\right]|L_2([0,T],\mathcal{K}^{2m(\gamma_m-k)}_{2,a+2m(\gamma_m-k)}(D))\right\| \notag \\
& \qquad \qquad \qquad  + \sum_{k=0}^{\gamma_m+1}
\left\|\partial_{t^{k}}\left[(u-v)\sum_{j=0}^{M-1} u^jv^{M-1-j}\right]|
L_2(D_T)\right\| \Bigg).  
\label{est-1}
\end{align}
Concerning the derivatives, we use Leibniz's formula twice and  we see that 
\begin{align}
\partial_{t^k}(u^M-v^M)
&=\partial_{t^k}\left[(u-v)\sum_{j=0}^{M-1} u^jv^{M-1-j}\right]\notag\\
&= \sum_{l=0}^k{k \choose l}\partial_{t^l}(u-v)\cdot \partial_{t^{k-l}}\left(\sum_{j=0}^{M-1}u^j v^{M-1-j}\right)\notag\\
&= \sum_{l=0}^k{k \choose l}\partial_{t^l}(u-v)\cdot 
\left[\left(\sum_{j=0}^{M-1} \sum_{r=0}^{k-l} {{k-l}\choose r} \partial_{t^r}u^j \cdot \partial_{t^{k-l-r}}v^{M-1-j}\right)\right]. \notag\\\label{est-2}
\end{align}
In order to estimate the terms $\partial_{t^r}u^j$ and $\partial_{t^{k-l-r}}v^{M-1-j}$ we apply Fa\`{a} di Bruno's formula \index{Fa\`{a} di Bruno formula}
\begin{equation}\label{FaaDiBruno}
\partial_{t^r}(f\circ g)=\sum\frac{r!}{k_1!\ldots k_r!}\left(\partial_{t^{k_1+\ldots + k_r}}f\circ g\right)\prod_{{i}=1}^{r}\left(\frac{\partial_{t^{{i}}}g}{{i}!}\right)^{k_{{i}}},
\end{equation}
where  the sum runs over all $r$-tuples of nonnegative integers $(k_1,\ldots, k_r)$ satisfying 
\begin{equation}\label{cond-kr}
1\cdot k_1+2\cdot k_2+\ldots +r\cdot k_r=r.
\end{equation}
{In particular, from \eqref{cond-kr}  we see that $k_{r}\leq 1$, where  $r=1,\ldots, k$. Therefore, the highest derivative $\partial_{t^r}u$ appears at most once.  }
We {apply the formula to} $g=u$ and $f(x)=x^j$ and
  make use of the embeddings  \eqref{kondratiev-emb} and the pointwise multiplier results from Theorem \ref{thm-pointwise-mult-2} (i) for $k\leq \gamma_m-1$.   (Note that the restriction '$a>\frac dp$' for $d=3$ in Theorem \ref{thm-pointwise-mult-2} (i) is satisfied  since in our situation  we have    $a+2m\geq m>\frac d2$ from  the  assumptions $a\in [-m,m]$ and $m\geq 2$.)
This yields 
\begin{align}
\Big\|&\partial_{t^r}u^j  \left. | \mathcal{K}^{2m(\gamma_m-k)}_{2,a+2m(\gamma_m-k)}(D)\right\| \notag\\
&\leq  c_{r,j}\left\|\sum_{k_1+\ldots +k_r\leq j, \atop 1\cdot k_1+2\cdot k_2+\ldots +r\cdot k_r=r} u^{j-(k_1+\ldots +k_r)}\prod_{{i}=1}^r \left|\partial_{t^{{i}}}u\right|^{k_{{i}}}| \mathcal{K}^{2m(\gamma_m-k)}_{2,a+2m(\gamma_m-k)}(D)\right\| \notag\\
&\lesssim \sum_{k_1+\ldots +k_r\leq j, \atop 1\cdot k_1+2\cdot k_2+\ldots +r\cdot k_r=r} \left\| u| \mathcal{K}^{2m(\gamma_m-k)}_{2,a+2m(\gamma_m-k)}(D)\right\|^{j-(k_1+\ldots +k_r)} 
\prod_{{i}=1}^{r} \left\| \partial_{t^{{i}}}u| \mathcal{K}^{2m(\gamma_m-k)}_{2,a+2m(\gamma_m-k)}(D)\right\|^{k_{{i}}}.  
\label{est-3}
\end{align}
For $k=\gamma_m$ we use Theorem \ref{thm-pointwise-mult-2}(ii). (Note that in Theorem \ref{thm-pointwise-mult-2}(ii) we require   that '$a-1\geq \frac dp-2$' with $d=3$ for the parameter. In our situation below $a-1$ has to be  replaced by $a$, which  leads to our restriction $a\geq \frac d2-2=-\frac 12$.)  Similar as above we obtain 
\begin{align}
\Big\|&\partial_{t^r}u^j  \left. | \mathcal{K}^{0}_{2,a}(D)\right\| \notag\\
&\leq  c_{r,j}\left\|\sum_{k_1+\ldots +k_r\leq j, \atop 1\cdot k_1+2\cdot k_2+\ldots +r\cdot k_r=r} u^{j-(k_1+\ldots +k_r)}\prod_{{i}=1}^r \left|\partial_{t^{{i}}}u\right|^{k_{{i}}}| \mathcal{K}^{0}_{2,a}(D)\right\| \notag\\
&\lesssim \sum_{k_1+\ldots +k_r\leq j, \atop 1\cdot k_1+2\cdot k_2+\ldots +r\cdot k_r=r} \left\| u| \mathcal{K}^{2}_{2,a+2}(D)\right\|^{j-(k_1+\ldots +k_r)}  
\left\| \partial_{t^r}u| \mathcal{K}^{0}_{2,a}(D)\right\|^{k_r}\prod_{{i}=1}^{r-1} \left\| \partial_{t^{{i}}}u| \mathcal{K}^{2}_{2,a+2}(D)\right\|^{k_{{i}}}\notag\\
&\lesssim \sum_{k_1+\ldots +k_r\leq j, \atop 1\cdot k_1+2\cdot k_2+\ldots +r\cdot k_r=r} \left\| u| \mathcal{K}^{2m\gamma_m}_{2,a+2m\gamma_m}(D)\right\|^{j-(k_1+\ldots +k_r)} \notag\\
& \qquad \qquad   
\left\| \partial_{t^r}u| \mathcal{K}^{2m(\gamma_m-r)}_{2,a+2m(\gamma_m-r)}(D)\right\|^{k_r}\prod_{{i}=1}^{r-1} \left\| \partial_{t^{{i}}}u| \mathcal{K}^{2m(\gamma_m-{i})}_{2,a+2m(\gamma_m-{i})}(D)\right\|^{k_{{i}}}. 
\label{est-33a}
\end{align}
Note that we require $\gamma_m\geq 1$ in the last step.  
We proceed similarly  for $\partial_{t^{k-l-r}}v^{M-1-j}$. 
Now \eqref{est-2} together with \eqref{est-3}  {and \eqref{est-33a}} inserted in \eqref{est-1} together with  Theorem \ref{thm-pointwise-mult-2} give \\

$\|\tilde{L}^{-1}\|\|u^M-v^M|\D_1\|$
\begin{align}
&\lesssim  \|\tilde{L}^{-1}\|\sum_{k=0}^{\gamma_m}\left(\int_0^T\left\|\partial_{t^k}\left[(u-v)\sum_{j=0}^{M-1} u^jv^{M-1-j}\right]|\mathcal{K}^{2m(\gamma_m-k)}_{2,a+2m(\gamma_m-k)}(D)\right\|^2\ud t\right)^{1/2}\notag\\
& {
\lesssim \|\tilde{L}^{-1}\|\sum_{k=0}^{\gamma_m}\sum_{l=0}^k\sum_{j=0}^{M-1}\sum_{r=0}^{k-l}\Bigg(\int_0^T \left\|\partial_{t^l}(u-v)
|\mathcal{K}^{2m(\gamma_m-k)}_{2,a+2m(\gamma_m-k)}(D)\right\|^2 }\notag \\
& {
\qquad \qquad 
\left\|\partial_{t^r}u^j |\mathcal{K}^{2m(\gamma_m-k)}_{2,a+2m(\gamma_m-k)}(D)\right\|^2
\left\|\partial_{t^{k-l-r}}v^{M-1-j}|\mathcal{K}^{2m(\gamma_m-k)}_{2,a+2m(\gamma_m-k)}(D)\right\|^2
\ud t\Bigg)^{1/2}
}\label{k=gamma_m} \\
& {
\lesssim \|\tilde{L}^{-1}\|\sum_{k=0}^{\gamma_m}\sum_{l=0}^k\sum_{j=0}^{M-1}\sum_{r=0}^{k-l}\Bigg(\int_0^T \left\|\partial_{t^l}(u-v)
|\mathcal{K}^{2m(\gamma_m-k)}_{2,a+2m(\gamma_m-k)}(D)\right\|^2 }\notag \\
& {
\sum_{\kappa_1+\ldots+\kappa_r\leq j, \atop \kappa_1+2\kappa_2+\ldots+r\kappa_r=r}
\left\|u |\mathcal{K}^{2m(\gamma_m-k)}_{2,a+2m(\gamma_m-k)}(D)\right\|^{2(j-(\kappa_1+\ldots+\kappa_r))}
\prod_{i=0}^r \left\| \partial_{t^{{i}}}u| \mathcal{K}^{2m(\gamma_m-{i})}_{2,a+2m(\gamma_m-{i})}(D)\right\|^{2\kappa_{{i}}}
}\notag \\
& {
\sum_{\kappa_1+\ldots+\kappa_{k-l-r}\leq M-1-j, \atop \kappa_1+2\kappa_2+\ldots+(k-l-r)\kappa_{k-l-r}=k-l-r}
\left\|v |\mathcal{K}^{2m(\gamma_m-k)}_{2,a+2m(\gamma_m-k)}(D)\right\|^{2(M-1-j-(\kappa_1+\ldots+\kappa_{k-l-r}))} }\notag \\
&{\qquad \qquad 
\prod_{i=0}^{k-l-r} \left\| \partial_{t^{{i}}}v| \mathcal{K}^{2m(\gamma_m-{i})}_{2,a+2m(\gamma_m-{i})}(D)\right\|^{2\kappa_{{i}}}
\ud t\Bigg)^{1/2}
}\notag \\
& {
\lesssim \|\tilde{L}^{-1}\|\sum_{k=0}^{\gamma_m}M\Bigg(\int_0^T \left\|\partial_{t^k}(u-v)
|\mathcal{K}^{2m(\gamma_m-k)}_{2,a+2m(\gamma_m-k)}(D)\right\|^2 }\notag \\
& {
\qquad 
\sum_{\kappa_1'+\ldots+\kappa_k'\leq \min\{M-1,k\}, \atop \kappa_k'\leq 1}
\max_{w\in \{u,v\}}\left\|w |\mathcal{K}^{2m(\gamma_m-k)}_{2,a+2m(\gamma_m-k)}(D)\right\|^{2(M-1-(\kappa_1'+\ldots+\kappa_k'))}}\notag \\
& {\qquad 
\prod_{i=0}^k  \max\left\{\left\| \partial_{t^{{i}}}u| \mathcal{K}^{2m(\gamma_m-{i})}_{2,a+2m(\gamma_m-{i})}(D)\right\|, \left\| \partial_{t^{{i}}}v| \mathcal{K}^{2m(\gamma_m-{i})}_{2,a+2m(\gamma_m-{i})}(D)\right\|, 1\right\}^{4\kappa_i'} \ud t\Bigg)^{1/2}
}\label{kappa} \\
&\lesssim M \|\tilde{L}^{-1}\| \cdot 
\left\|u-v| \bigcap_{k=0}^{\gamma_m+1}H^{k}([0,T],\mathcal{K}^{2m(\gamma_m-(k-1))}_{2,a+2m(\gamma_m-(k-1))}(D))\right\|\cdot \notag\\
& \qquad \max_{w\in \{u,v\}}\max_{l=0,\ldots, \gamma_m} \max \Big(\left\| \partial_{t^l}w |L_{\infty}([0,T],\mathcal{K}^{2m(\gamma_m-l)}_{2,a+2m(\gamma_m-l)}(D))\right\|,\  
1\Big)^{{2(M-1)}}.\notag\\ \label{est-4}
\end{align}
We give some explanations concerning the estimate above. In \eqref{k=gamma_m} the term with $k=\gamma_m$ requires some special care since we have to apply Theorem \ref{thm-pointwise-mult-2} (ii). In this case we calculate 
\begin{align*}
\Bigg\| & \left.\partial_{\gamma_m}\left[(u-v)\left(\sum_{j=0}^{M-1}u^jv^{M-1-j}\right)\right]|\calk^{0}_{2,a}(D)\right\|\notag \\ 
& 
\lesssim \left\|\partial_{\gamma_m}(u-v)|\calk^{0}_{2,a}(D)\right\|
\sum_{j=0}^{M-1}\left\|u^jv^{M-1-j}|\calk^{2}_{2,a+2}(D)\right\|  \notag\\ 
& \qquad + \left\|u-v|\calk^{2}_{2,a+2}(D)\right\|
\sum_{j=0}^{M-1}\sum_{r=0}^{\gamma_m}\left\|(\partial_{t^r}u^j)(\partial_{t^{\gamma_m-r}}v^{M-1-j})|\calk^{0}_{2,a}(D)\right\| \\ 
& \qquad 
+ \left\|\sum_{r=1}^{\gamma_m-1}{\gamma_m\choose r}\partial_r(u-v)\partial_{\gamma_m-r}\left(\sum_{j=0}^{M-1}\ldots\right) |\calk^{0}_{2,a}(D)\right\|.  
\end{align*}
The lower order derivatives in the last line  cause no problems since we can (again) apply  Theorem \ref{thm-pointwise-mult-2}(i) as before. 
The term $\left\|u^jv^{M-1-j}|\calk^{2}_{2,a+2}(D)\right\|$ can now be further estimated with the help of Theorem \ref{thm-pointwise-mult-2}(i). For the term $\sum_{r=0}^{\gamma_m}\left\|(\partial_{t^r}u^j)(\partial_{t^{\gamma_m-r}}v^{M-1-j})|\calk^{0}_{2,a}(D)\right\|$ we again use Theorem \ref{thm-pointwise-mult-2}(ii), then proceed as in \eqref{est-33a} and see that the resulting estimate yields  \eqref{k=gamma_m}.\\ 
Moreover, in \eqref{kappa} we used the fact that in the step before we have two sums with  $\kappa_1+\ldots +\kappa_r\leq j$ and $\kappa_1+\ldots+\kappa_{k-l-r}\leq M-1-j$, i.e., we have $k-l$ different $\kappa_i$'s which leads to at most $k$ different $\kappa_i$'s if $l=0$.  We allow all combinations of $\kappa_i$'s and  redefine the $\kappa_i$'s in the second sum leading to $\kappa_1', \ldots , \kappa_k'$ with $\kappa_1'+\ldots+\kappa_k'\leq M-1$ and replace the old conditions $\kappa_1+2\kappa_2+r\kappa_r\leq r$ and $\kappa_1+2\kappa_2+(k-l-r)\kappa_{k-l-r}\leq k-l-r$ by the weaker ones $\kappa_1'+\ldots+\kappa_k'\leq k$ and $\kappa_k'\leq 1$. This causes no problems since the other terms appearing in this step do not depend on $\kappa_i$ apart from the product term. There, the fact that some of the old $\kappa_i$'s from both sums might coincide is reflected in the new exponent $4\kappa_i'$.   
From Theorem \ref{thm-sob-emb}  we conclude that 
\begin{eqnarray}
u,v \in S&\hookrightarrow & \bigcap_{k=0}^{\gamma_m+1}H^{{k}}([0,T],\mathcal{K}^{2m(\gamma_m-(k-1))}_{2,a+2m(\gamma_m-(k-1))}(D))\notag\\
&\hookrightarrow & \bigcap_{k=1}^{\gamma_m+1}\mathcal{C}^{{k-1,\frac 12}}([0,T],\mathcal{K}^{2m(\gamma_m-(k-1))}_{2,a+2m(\gamma_m-(k-1))}(D))
\notag\\
&\hookrightarrow&  \bigcap_{k=1}^{\gamma_m+1}{C}^{{k-1}}([0,T],\mathcal{K}^{2m(\gamma_m-(k-1))}_{2,a+2m(\gamma_m-(k-1))}(D))
= \bigcap_{l=0}^{\gamma_m}{C}^{{l}}([0,T],\mathcal{K}^{2m(\gamma_m-l)}_{2,a+2m(\gamma_m-l)}(D)), \notag
\end{eqnarray}
hence, the term  {involving the maxima, $\max_{w\in \{u,v\}}\max_{l=0,\ldots, \gamma_m}\max (\ldots)^{M-1}$} in \eqref{est-4} is bounded {by $\max(R+\|\tilde{L}^{-1}f|S\|,1)^{M-1}$}.  
Moreover, since $u$ and $v$ are taken from $B_R(\tilde{L}^{-1}f)$  in {$S=S_1\cap  S_2$}, we obtain from \eqref{est-4},  
 \begin{align}
\|\tilde{L}^{-1}\| & \|u^M-v^M|\D_1\|\notag \\
&\leq  {c_0}\|\tilde{L}^{-1}\|M\max(R+\|\tilde{L}^{-1}f|S\|,1)^{{2(M-1)}}\|u-v| S\|\notag\\
&\leq  {c_2}\|\tilde{L}^{-1}\|M\max(R+\|\tilde{L}^{-1}\|\cdot \|f|\D\|,1)^{{2(M-1)}} \|u-v| S\|\notag\\
&= {c_2}\|\tilde{L}^{-1}\|M\max(R+\|\tilde{L}^{-1}\| \eta,1 )^{{2(M-1)}} \|u-v| S\|,   \label{est-ball}
\end{align}
where we put $\eta:=\|f|\D\|$ in the last line, $c_0$ denotes the constant resulting from \eqref{est-3} and \eqref{est-4} and $c_2=c_0c_1$ with $c_1$ being the constant from the estimates in Theorem \ref{thm-weighted-sob-reg}. 

We now turn our attention towards the second term $\|\tilde{L}^{-1}\| \|u^M-v^M|\D_2\|$ in \eqref{est-1} and calculate \\
{
\begin{align}
\|\tilde{L}^{-1}\|& \|(u^M-v^M)|\D_2\| \notag\\
&= \|\tilde{L}^{-1}\|\left\|(u-v)\sum_{j=0}^{M-1} u^jv^{M-1-j}|H^{\gamma_m+1}([0,T],L_2(D))\right\|\notag\\
&= \|\tilde{L}^{-1}\|\sum_{k=0}^{\gamma_m+1}\left\|\partial_{t^k}\left[(u-v)\sum_{j=0}^{M-1} u^jv^{M-1-j}\right]|L_2(D_T)\right\| \notag\\
&=  \|\tilde{L}^{-1}\|\sum_{k=0}^{\gamma_m+1}\left\|\sum_{l=0}^k{k \choose l}\partial_{t^l}(u-v)\cdot \right. 
\left.\left[\left(\sum_{j=0}^{M-1} \sum_{r=0}^{k-l} {{k-l}\choose r} \partial_{t^{r}}u^j \cdot \partial_{t^{k-l-r}}v^{M-1-j}\right)\right]|L_2(D_T)\right\| \notag\\
&\lesssim  \|\tilde{L}^{-1}\|\sum_{k=0}^{\gamma_m+1}\left\|\sum_{l=0}^k|\partial_{t^l}(u-v)|\cdot \right.
\left.\left[\left(\sum_{j=0}^{M-1} \sum_{r=0}^{k-l}  |\partial_{t^{r}}u^j \cdot \partial_{t^{k-l-r}}v^{M-1-j}|\right)\right]|L_2(D_T)\right\|, 
\label{est-1a}
\end{align}
where we used Leibniz's formula twice as in \eqref{est-2} in the second but last line. Again  Fa\`{a} di Bruno's formula, cf. \eqref{FaaDiBruno}, is applied  in order to estimate the derivatives in \eqref{est-1a}. We use a special case of the multiplier result from \cite[Sect.~4.6.1, Thm.~1(i)]{RS96}, 
which tells us that for  $m>\frac 32$ we have 
\begin{equation}\label{multiplier-lim}
\|uv|L_2\|\lesssim \|u|H^m\|\cdot \|v|L_2\|   
\end{equation}
(we remark that this is exactly the point where our assumption $m\geq 2$ comes into play). 
With this we obtain 
\begin{align}
\Big\|  \partial_{t^r}u^j  | L_2(D)\Big\| 
& \leq  c_{r,j}\left\|\sum_{k_1+\ldots +k_r\leq j} u^{j-(k_1+\ldots +k_r)}\prod_{{i}=1}^r \left|\partial_{t^{{i}}}u\right|^{k_{{i}}}| L_2(D)\right\| \notag\\
&\lesssim \sum_{k_1+\ldots +k_r\leq j} \left\| u| H^m(D)\right\|^{j-(k_1+\ldots +k_r)} \prod_{{i}=1}^{r-1} \left\| \partial_{t^{{i}}}u| H^m(D)\right\|^{k_{{i}}} \left\| \partial_{t^r}u| L_2(D)\right\|^{k_r}.  \label{est-3a}
\end{align}
Similar  for $\partial_{t^{k-l-r}}v^{M-1-j}$. As before, from \eqref{cond-kr}  we observe {$k_{r}\leq 1$, therefore the highest derivative $u^{(r)}$} appears at most once.  {Note that since $H^m(D)$ is a multiplication algebra for $m> \frac d2$, we get \eqref{est-3a} with $L_2(D)$ replaced by $H^m(D)$ as well.} 
Now {\eqref{multiplier-lim} and \eqref{est-3a}  inserted in \eqref{est-1a}} gives 
\begin{align}
\|& \tilde{L}^{-1}\| \|u^M-v^M|\D_2\|\notag\\
& {
= \|\tilde{L}^{-1}\|\sum_{k=0}^{\gamma_m+1}\Bigg(\int_0^T\left\|\partial_{t^k}(u-v)\sum_{j=0}^{M-1}u^jv^{M-1-j}| L_2(D)\right\|^2 \ud t\Bigg)^{1/2}
}\notag\\
& {
\lesssim \|\tilde{L}^{-1}\|\sum_{k=0}^{\gamma_m+1}\sum_{l=0}^k\Bigg(\int_0^T\left\|\partial_{t^l}(u-v)|H^m(D)\right\|^2 }\notag\\
& \qquad\qquad  {\sum_{j=0}^{M-1}\sum_{r=0}^{k-l}\left\|\partial_{t^r}u^j \cdot \partial_{t^{k-l-r}}v^{M-1-j}| L_2(D)\right\|^2 \ud t\Bigg)^{1/2}
}\notag\\
& {
\lesssim \|\tilde{L}^{-1}\|\sum_{k=0}^{\gamma_m+1}\sum_{l=0}^k\Bigg(\int_0^T\bigg\{\left\|\partial_{t^l}(u-v)|H^m(D)\right\|^2 }\notag\\
& \qquad\qquad  {\sum_{j=0}^{M-1}\sum_{r=0, \atop (k-l-r\neq \gamma_m+1)\wedge (r\neq \gamma_m+1)}^{k-l}\left\|\partial_{t^r}u^j|H^m(D)\|^2 \|\partial_{t^{k-l-r}}v^{M-1-j}| H^m(D)\right\|^2 }\notag \\
& {
\qquad \qquad +\|u-v|H^m(D)\|^2\|\partial_{t^{\gamma_m+1}}u^j|L_2(D)\|^2\|v^{M-1-j}|H^m(D)\|^2}\notag \\
& \qquad \qquad {+ \|u-v|H^m(D)\|^2\|u^j|H^m(D)\|^2\|\partial_{t^{\gamma_m+1}}v^{M-1-j}|L_2(D)\|^2
\bigg\}\ \ud t\Bigg)^{1/2}
}\notag\\
&{\lesssim \|\tilde{L}^{-1}\|\sum_{k=0}^{\gamma_m+1}\sum_{l=0}^k\Bigg(\int_0^T\left\|\partial_{t^l}(u-v)| H^m(D)\right\|^2 \cdot }\notag \\
& {\qquad \qquad \sum_{j=0}^{M-1} \sum_{r=0}^{k-l}\sum_{\kappa_1+\ldots +\kappa_{r}\leq j, \atop \kappa_1+2\kappa_2+\ldots+ r\kappa_{r}\leq r}
  \left\| u| H^m(D)\right\|^{2(j-(\kappa_1+\ldots +\kappa_{r}))}} \notag \\
   & { \qquad 
\left.\begin{cases}   \left\| \partial_{t^{{r}}}u| L_2(D)\right\|^{2\kappa_{{r}}}\prod_{{i}=1}^{r-1} \left\| \partial_{t^{{i}}}u| H^m(D)\right\|^{2\kappa_{{i}}}, & r=\gamma_m+1,\\
\prod_{{i}=1}^{r} \left\| \partial_{t^{{i}}}u| H^m(D)\right\|^{2\kappa_{{i}}},    & r\neq \gamma_m+1
 \end{cases}  \right\}
}\notag \\
  & {\qquad  \sum_{\kappa_1+\ldots +\kappa_{k-l-r}\leq M-1-j, \atop \kappa_1+2\kappa_2+\ldots+(k-l-r)\kappa_{k-l-r}\leq k-l-r} 
  \left\| v| H^m(D)\right\|^{2(M-1-j-(\kappa_1+\ldots +\kappa_{k-l-r}))} }\notag \\
    & {\qquad 
\left.\begin{cases}   \left\| \partial_{t^{{r}}}v| L_2(D)\right\|^{2\kappa_{{r}}}\prod_{{i}=1}^{k-l-r-1} \left\| \partial_{t^{{i}}}v| H^m(D)\right\|^{2\kappa_{{i}}}, & k-l-r=\gamma_m+1,\\
\prod_{{i}=1}^{l-k-r} \left\| \partial_{t^{{i}}}v| H^m(D)\right\|^{2\kappa_{{i}}},    & k-l-r\neq \gamma_m+1
 \end{cases}  \right\}
\ud t\Bigg)^{1/2}}\notag\\
&{\lesssim \|\tilde{L}^{-1}\|\sum_{k=0}^{\gamma_m+1}\Bigg(\int_0^T\left\|\partial_{t^k}(u-v)| H^m(D)\right\|^2 \cdot }\notag \\
& {\qquad \qquad M \sum_{\kappa_1'+\ldots +\kappa_{k}'\leq \min\{M-1,k\}}
  \max_{w\in \{u,v\}}\left\| w| H^m(D)\right\|^{2(M-1-(\kappa_1'+\ldots +\kappa_{k}'))}} \notag \\
   & { \qquad 
\left.\begin{cases}   \max(\left\| \partial_{t^{{k}}}w| L_2(D)\right\|^{4\kappa_{k}'}\prod_{{i}=1}^{k-1} \left\| \partial_{t^{{i}}}w| H^m(D)\right\|^{4\kappa_{i}'}, 1), & k=\gamma_m+1,\\
\max(\prod_{{i}=1}^{k} \left\| \partial_{t^{{i}}}w| H^m(D)\right\|^{4\kappa_{i}'}, 1),    & k\neq \gamma_m+1
 \end{cases}  \right\}
\ud t\Bigg)^{1/2}}\notag\\
&\lesssim \|\tilde{L}^{-1}\|M
\|u-v| H^{\gamma_m+1}([0,T],H^m(D))\|^2\max_{w\in \{u,v\}}\max_{{i}=0,\ldots, \gamma_m} \max\notag\\
&  \qquad \left(
\left\| \partial_{t^{{i}}}w|L_{\infty}([0,T],H^m(D))\right\|,\  \left\| \partial_{t^{\gamma_m+1}}w|L_{\infty}([0,T],L_2(D))\right\|,\ 1\right)^{{{2(M-1)}}}.\notag\\ \label{est-4a}
\end{align}
{Similar to \eqref{est-4} in the calculations above the term $k=\gamma_m+1$ required some special care. For the redefinition of the $\kappa_i$'s in the second but last line in \eqref{est-4a} we refer to the explanations given after \eqref{est-4}. }
From Theorem \ref{thm-sob-emb} we see that 
\begin{eqnarray}
u,v \in S&\hookrightarrow & H^{{\gamma_m+1}}([0,T],{\mathring{H}^m}(D))\cap H^{{\gamma_m+2}}([0,T],L_2(D))
\notag\\
&\hookrightarrow & \mathcal{C}^{{\gamma_m,\frac 12}}([0,T],{\mathring{H}^m}(D))\cap \mathcal{C}^{{\gamma_m+1,\frac 12}}([0,T],L_2(D))
\notag\\
&\hookrightarrow & {C}^{{\gamma_m}}([0,T],{\mathring{H}^m}(D))
\cap {C}^{{\gamma_m+1}}([0,T],L_2(D)), 
\label{est-4aa}
\end{eqnarray}
hence the term  ${\max_{w\in \{u,v\}}\max_{m=0,\ldots, l}\max(\ldots)^{M-1}}$ in \eqref{est-4a} is bounded.  Moreover, since $u$ and $v$ are taken from $B_R(\tilde{L}^{-1}f)$  in $S_2={H}^{m,\gamma_m+2\ast}(D_T)=H^{\gamma_m+1}([0,T],{\mathring{H}^m}(D))\cap  H^{\gamma_m+2}([0,T],H^{-m}(D)) $, as in \eqref{est-ball} we obtain from \eqref{est-4a} {and \eqref{est-4aa}},  
 \begin{align}
\|\tilde{L}^{-1}\|\|u^M-v^M|\D_2\|\leq c_3 
 \|\tilde{L}^{-1}\|M\max(R+\|\tilde{L}^{-1}\| \eta, 1)^{{2(M-1)}}\cdot \|u-v| S\|, \label{est-ball_a}
\end{align}
where we put $\eta:=\|f|\D\|$ {and $c_3$ denotes the constant arising from our estimates \eqref{est-4a} and \eqref{est-4aa} above}. 
}
Now \eqref{est-1} together with \eqref{est-ball} and \eqref{est-ball_a} yields 
\begin{align}\label{est-ab}
\|\tilde{L}^{-1}(u^M-v^M)|S\|
& \leq \|\tilde{L}^{-1}\|\|(u^M-v^M)|\widetilde{\mathcal{D}}\|\notag \\
&\leq c \|\tilde{L}^{-1}\|M\max(R+\|\tilde{L}^{-1}\|\eta, 1)^{M-1}\|u-v|S\|, 
\end{align}
where $c=c_2+c_3$. 
For $\tilde{L}^{-1}\circ N$ to be a contraction, we therefore require 
\[
{c}\varepsilon \|\tilde{L}^{-1}\|M\max(R+\|\tilde{L}^{-1}\|\eta,1)^{{2(M-1)}}<1,
\]
{cf. \eqref{est-0}.} In case of $\ \max(R+\|\tilde{L}^{-1}\|\eta,1)=1$ this leads to 
\begin{equation}\label{cond-01}
\|\tilde{L}^{-1}\|<\frac{1}{{c}\varepsilon M}.
\end{equation}
On the other hand, if  $\ \max(R+\|\tilde{L}^{-1}\|\eta,1)=R+\|\tilde{L}^{-1}\|\eta$, we choose  $R=(r_0-1)\eta\|\tilde{L}^{-1}\|$, which gives {rise to} the condition 
\begin{equation}\label{cond-1}
{c}\varepsilon\|\tilde{L}^{-1}\|M(r_0\|\tilde{L}^{-1}\|\eta)^{{2(M-1)}}<1,\quad {\text{i.e.,}} \quad \eta^{{2(M-1)}} \|\tilde{L}^{-1}\|^{{2M-1}}<\frac{1}{{c}\varepsilon M}\left(\frac{1}{r_0}\right)^{{2(M-1)}}.
\end{equation}
{\em Step 2: } The calculations in Step 1 show that $u^M\in \widetilde{\mathcal{D}}$: 
The fact that $u^M\in \D_1\cap \D_2$ follows from the estimate \eqref{est-ab}.   In particular, taking $v=0$ in \eqref{est-ab} we get an estimate from above for $\|u^M|\widetilde{\mathcal{D}}\|$. The upper bound depends on $\|u|S\|$ and several constants which depend on $u$ but are finite whenever we have $u\in S$, see also \eqref{est-4} and \eqref{est-4a}. The dependence on $R$ in \eqref{est-ab} comes from the fact that we choose $u\in B_R(\tilde{L}^{-1}f)$ in $S$ there. 
However, the same argument can also be applied to an arbitrary $u\in S$; this would result in a different constant $\tilde{c}$. 
In order to have $u^M\in \widetilde{\mathcal{D}}$, we still need to show that $\mathrm{Tr}\left(\partial_{t^k}u^M\right)=0$, $k=0,\ldots, \gamma_m$. This follows from the same arguments as  in   \cite[Thm.~4.10]{DS19}:  
Since $u\in S\hookrightarrow H^{\gamma_m+2}([0,T], H^{-m}(D))\hookrightarrow C^{\gamma_m+1}([0,T], H^{-m}(D))$ we see that the trace operator $\mathrm{Tr}\left(\partial_{t^k}u\right):=\left(\partial_{t^k}u\right)(0,\cdot)$ is well defined for $k=0,\ldots, \gamma_m+1$.  
Using the initial assumption $u(0,\cdot)=0$ in Problem \ref{prob_nonlin}, 
by density arguments ($C^{\infty}(D_T)$ is dense in $S$)
and  induction we deduce  that $(\partial_{t^k}u)(0,\cdot)=0$ for all $k=0,\ldots, \gamma_m+1$. 
 Moreover, since by Theorem \ref{thm-sob-emb} 
 \begin{align*}
 u^M \in \D_1\cap \D_2 &\hookrightarrow H^{\gamma_m+1}([0,T], L_2(D)) 
 \hookrightarrow C^{\gamma_m}([0,T], L_2(D)), 
\end{align*}
 we see that the trace operator $\mathrm{Tr}\left(\partial_{t^k}u^M\right):=\left(\partial_{t^k}u^M\right)(0,\cdot)$ is well defined for $k=0,\ldots, \gamma_m$. By \eqref{est-3a} below the term $\|\left(\partial_{t^k}u^M\right)(0,\cdot)|L_2(D)\|$ is estimated from above by {powers of} $\|\left(\partial_{t^l}u\right)(0,\cdot)|H^m(D)\|$, $l=0,\ldots, k$. 
 Since all these terms are equal to zero, 
 this shows that $u^M\in \widetilde{\mathcal{D}}$. \\
{\em Step 3: } The next step is to show that $(\tilde{L}^{-1}\circ N)(B_R(\tilde{L}^{-1}f))\subset B_R(\tilde{L}^{-1}f)$ in $S$. Since $(\tilde{L}^{-1}\circ N)(0)=\tilde{L}^{-1}(f-\varepsilon 0^M)=\tilde{L}^{-1}f$, we only need to apply the above estimate \eqref{est-ab} with $v=0$. This gives 
\begin{align*}
\varepsilon\|\tilde{L}^{-1}u^M|S\|
&\leq {c}\varepsilon \|\tilde{L}^{-1}\|M\max(R+\|\tilde{L}^{-1}\| \eta,1)^{{2(M-1)}}(R+\|\tilde{L}^{-1}\| \eta)\\
&\overset{!}{\leq}R=(r_0-1)\eta\|\tilde{L}^{-1}\|, 
\end{align*}
which, {in case that}  $\max(R+\|\tilde{L}^{-1}\|\eta,1)=1$,  leads to 
\begin{equation}\label{cond-02}
\|\tilde{L}^{-1}\|<\frac{r_0-1}{r_0}\left(\frac{1}{{c}\varepsilon M}\right), 
\end{equation}
whereas for $\max(R+\|\tilde{L}^{-1}\|\eta,1)=R+\|\tilde{L}^{-1}\|\eta$ we get 
\begin{equation}\label{cond-2}
{\eta^{2(M-1)} \|\tilde{L}^{-1}\|^{2M-1}\leq \frac{1}{{c}\varepsilon M}(r_0-1)\left(\frac{1}{r_0}\right)^{2M-1}. }
\end{equation}
We see that condition \eqref{cond-02} implies \eqref{cond-01}. Furthermore, since 
\[
{(r_0-1)\left(\frac{1}{r_0}\right)^{2M-1}=\frac{r_0-1}{r_0}\left(\frac{1}{r_0}\right)^{2(M-1)}<\left(\frac{1}{r_0}\right)^{2(M-1)}, }
\]
also condition \eqref{cond-2} implies \eqref{cond-1}. 
Thus, by applying Banach's fixed point theorem in a sufficiently small ball around the solution of the corresponding linear problem, we obtain {a unique} solution of Problem \ref{prob_nonlin}. 
\end{proof}

\remark{
The restriction $m\geq 2$ in Theorem \ref{nonlin-B-reg1} comes from the fact that we require $s_2=m>\frac d2=\frac 32$ in \eqref{multiplier-lim}. This assumption can probably be weakened, since we expect   the solution to satisfy  $u\in L_2([0,T], H^{s}(D))$ for all $s<\frac 32$, see also Remark \ref{gen-thm-parab-Besov} and the explanations given there.\\ 
Moreover, the restriction  $a\geq -\frac 12$ in  Theorem \ref{nonlin-B-reg1} comes from Theorem \ref{thm-pointwise-mult-2}(ii) that we applied. 
Together with the restriction $a\in [-m,m]$ we are looking for  $a\in [-\frac 12,m]$ if the domain is a corner domain, e.g.  a smooth cone $K\subset \real^3$ (subject to some truncation). For polyhedral cones with edges $M_k$, $k=1,\ldots, l$,  we   furthermore require $-\delta_-^{(k)}<a+2m(\gamma_m-i)+m<\delta^{(k)}_+$ for $i=0,\ldots, \gamma_m$  from Theorem \ref{thm-weighted-sob-reg}. 
}

\section{Regularity results in Besov spaces}
\label{sect-besov-reg}

With all  preliminary work, in this section we finally come to the presentation of the regularity results in Besov spaces for Problems \ref{prob_parab-1a} and \ref{prob_nonlin}. 
For this purpose, we rely  on the results from Section \ref{sect-reg-sob-kon} on regularity in Sobolev and Kondratiev spaces for the respective problems and combine these with the embeddings  of Kondratiev spaces into Besov spaces. It turns out  that in all cases studied  the Besov regularity is higher than the Sobolev regularity.  This indicates that adaptivity pays off when solving these problems numerically.  \\
The Sobolev regularity   we are working with (e.g. see Theorem \ref{Sob-reg-2} for Problem \ref{prob_parab-1a}) canonically comes out from the variational formulation of the problem, i.e., we have spatial Sobolev regularity $m$ if the corresponding differential operator is of order $2m$. We give an outlook on how our results could be improved by using regularity results in fractional Sobolev spaces instead. It is planned to do further investigations in this direction in the future. \\
Moreover, we discuss the role of the weight parameter $a$ appearing in our Kondratiev spaces to some extent. 

\subsection{Besov regularity of Problem I}
\label{subsect-besov-p1}

A combination of Theorem \ref{thm-weighted-sob-reg} (Kondratiev regularity A)  and the embedding in Theorem \ref{thm-hansen-gen}  yields the following Besov regularity of Problem \ref{prob_parab-1a}.

\begin{theorem}[{\bf Parabolic Besov regularity A}]\label{thm-parab-Besov}
Let $D$ be a bounded polyhedral domain in $\real^3$. Let $\gamma\in \nat $ with  ${\gamma\geq 2m}$ and put $\gamma_m:=\left[ \frac{\gamma-1}{2m}\right]$. Furthermore, let  $a\in \real$ with   $a\in [-m,m]$.  Assume that  the right-hand side $f$ of Problem \ref{prob_parab-1a} satisfies 
{
\begin{itemize}
\item[(i)] $\partial_{t^k} f\in L_2(D_T)\cap L_2([0,T],\mathcal{K}^{2m(\gamma_m-k)}_{2,a+2m(\gamma_m-k)}(D))$, \ $k=0,\ldots, \gamma_m$; \quad 
$\partial_{t^{\gamma_m+1}} f\in L_2(D_T)$. 
\item[(ii)] $\partial_{t^k} f(x,0)=0$, \quad  $k=0,1,\ldots, {\gamma_m}.$
\end{itemize}
}
{Furthermore, let  Assumption \ref{assumptions}  hold for weight parameters $b=a+2m(\gamma_m-i)$, where $i=0,\ldots, \gamma_m$, and  $b'=-m$.  }
Then 
for the weak solution $u\in {{H}}^{m,{\gamma_m+2}\ast}(D_T)$ of Problem \ref{prob_parab-1a}, we have 
\begin{equation}\label{parab-Besov}
u\in L_{2}([0,T],B^{\alpha}_{\tau,\infty}(D)) \qquad \text{for all}\quad {0<\alpha<\min\left(\gamma,\frac{3}{\delta}m\right), } \quad     
\end{equation}
where $\frac 12 <\frac{1}{\tau}<\frac{\alpha}{d}+\frac 12$ and $\delta$ denotes the dimension of the singular set of $D$. 
In particular, for any $\alpha$ satisfying \eqref{parab-Besov} and $\tau$ as above, we have the a priori estimate 
\begin{align*}
\|u|& L_{2}([0,T],B^{\alpha}_{\tau,\infty}(D))\|
\lesssim  
\sum_{k=0}^{{\gamma_m}}\|\partial_{t^k} f|{L_2([0,T],{\mathcal{K}^{2m(\gamma_m-k)}_{2,a+2m(\gamma_m-k)}(D))}}\|+\sum_{k=0}^{{\gamma_m}+1}\|\partial_{t^k} f|{L_2(D_T)}\|. 
\end{align*}
\end{theorem}

\begin{proof}
According to Theorem  \ref{thm-weighted-sob-reg} by our assumptions we know  $ u\in L_2([0,T], {\mathcal{K}^{2m(\gamma_m+1)}_{2,a+2m(\gamma_m+1)}(D)})$. Together with  Theorem  \ref{thm-hansen-gen} {(choosing $k=0$)} we obtain  
\begin{align*}
u\ \in \ &L_2([0,T],{\mathcal{K}^{2m(\gamma_m+1)}_{2,a+2m(\gamma_m+1)}(D)})\cap {{H}}^{m,{\gamma_m+2}\ast}(D_T)\\
&\hookrightarrow  L_2([0,T],{\mathcal{K}^{2m(\gamma_m+1)}_{2,a+2m(\gamma_m+1)}(D)})\cap L_2([0,T],H^m(D))\\
&\hookrightarrow   L_2([0,T],{\mathcal{K}^{2m(\gamma_m+1)}_{2,a+2m(\gamma_m+1)}(D)})\cap L_2([0,T],B^m_{2,\infty}(D))\\
&\hookrightarrow   L_2([0,T],\mathcal{K}^{\alpha}_{2,a+2m(\gamma_m+1)}(D)\cap B^m_{2,\infty}(D))
\hookrightarrow    L_2([0,T],B^{\alpha}_{\tau,\infty}(D)),   
\end{align*}
{where in the third step we use the fact that  {$ 2m(\gamma_m+1)\geq 2m\left(\frac{\gamma}{2m}-1+1\right)=\gamma$ and choose $\alpha\leq \gamma$.}}  Moreover,   the {condition} on $a$ from Theorem \ref{thm-hansen-gen} yields  
$ 
{m=}{\min(m,a+2m(\gamma_m+1))}>\frac{\delta}{3}\alpha.   
$ 
{Therefore, the upper bound for $\alpha$ is  $ \alpha<{\min\left(\gamma,\frac{3}{\delta}m\right)}. $ }
Concerning the restriction on $\tau$, Theorem  \ref{thm-hansen-gen} with $\tau_0=2$ gives 
$
\frac 12<\frac{1}{\tau}<\frac{1}{\tau^{\ast}}=\frac{\alpha}{3}+\frac 12.
$ 
This completes the proof. 
\end{proof}

\begin{rem}[{\bf The parameter $a$}]\label{discuss_a}
We discuss the role of the weight parameter in our Kondratiev spaces: Note that on the one hand we require $a+2m(\gamma_m+1)>0$ in order to apply the embedding from  Theorem \ref{thm-hansen-gen}. Since we assume $a\in [-m,m]$ this is always true. On the other hand it should be expected that the derivatives of the solution $u$ have singularities near the boundary of the polyhedral domain. Thus, looking at the highest derivative of $u(t)\in \calk^{2m(\gamma_m+1)}_{2,a+2m(\gamma_m+1)}(D)$ we see that we require 
\[
\sum_{|\alpha|=2m(\gamma_m+1)}\int_D \rho^{-ap}(x)|\partial^{\alpha}u(t,x)|^p\ud x<\infty, 
\]
hence, if $a<0$ the derivatives of the solution $u$ might be unbounded near  the boundary of $D$. From this it follows that the range 
$\ 
{-m}<a<0 
\ $ 
is the most interesting for our considerations. 
\end{rem}

\remark{\label{gen-thm-parab-Besov}
The above theorem relies on the fact that  Problem \ref{prob_parab-1a} has a weak solution $u\in H^{m,{\gamma_m+2}\ast}(D_T)={H^{\gamma_m+1}}([0,T], \mathring{H}^m(D))\cap {H^{\gamma_m+2}}([0,T],H^{-m}(D))\hookrightarrow L_2([0,T], H^m(D))$, cf. Theorem \ref{Sob-reg-3}.  We strongly believe that (in good  agreement with the elliptic case) this result can be improved by studying the regularity of Problem \ref{prob_parab-1a}  in fractional Sobolev spaces $H^s(D)$. 
In this case (assuming that the weak solution of Problem \ref{prob_parab-1a} satisfies $u\in L_2([0,T],H^s(D))$ for some $s>0$) under the assumptions of Theorem \ref{thm-parab-Besov}, using Theorem  \ref{thm-weighted-sob-reg} and Theorem \ref{thm-hansen-gen} {(with $k=0$)}, we {would} obtain 
\beq\label{emb-bes-sob-kon}
 u\in  L_2([0,T],\mathcal{K}^{\alpha}_{2,a'}(D))\cap L_2([0,T],H^s(D))\hookrightarrow L_2([0,T],B^{\alpha}_{\tau,\infty}(D)),
\eeq
where $a'=a+ 2m(\gamma_m+1)\geq a+2m$ and again $\frac 12<\frac{1}{\tau}<\frac{\alpha}{3}+\frac 12$ but the restriction on $\alpha$ now reads as 
\beq\label{rest-a-zusatz}
\alpha<\frac{3}{\delta}\min(s, a').  
\eeq
For general Lipschitz domains $D\subset \real^3$  we expect {that the solution of Problem \ref{prob_parab-1a} (for $m=1$) is contained in $H^s(D)$ for} all $s<\frac 32$ (as  in the elliptic case, cf. \cite{JK95}). This would   lead to $\alpha<\frac 92$ when $\delta=1$. For convex domains it probably even  holds that $s=2$ (for the heat equation this was already proven in \cite{Wo07}). 
First results in this direction can be found in \cite{DS18}. 
}

Alternatively,  we combine  Theorem \ref{thm-weighted-sob-reg-2} (Kondratiev regularity B) and Theorem \ref{thm-hansen-gen}. This leads to  the following  regularity result in Besov spaces.

{
\begin{theorem}[{\bf Parabolic Besov regularity B}]\label{thm-parab-Besov-2}
Let $D$ be a bounded polyhedral domain in $\real^3$. Let $\gamma\in \nat $ with  ${\gamma\geq 2m}$. Moreover, let  $a\in \real$ with   ${a\in [-m,m]}$.  Assume that  the right-hand side $f$ of Problem \ref{prob_parab-1a} satisfies 
\begin{itemize}
\item[(i)] $f\in \bigcap_{l=0}^{\infty}H^l([0,T],L_2(D)\cap \calk^{\gamma-2m}_{2,a}(D))$. 
\item[(ii)] $\partial_{t^l} f(x,0)=0$, \quad  $l\in \nat_0$. 
\end{itemize}
Furthermore, let  Assumption \ref{assumptions}  hold for weight parameters $b=a$ and  $b'=-m$.  
Then for the weak solution $\bigcap_{l=0}^{\infty}u\in {{H}}^{m,{l+1}\ast}(D_T)$ of Problem \ref{prob_parab-1a}, we have 
\begin{equation}\label{parab-Besov2}
 u\in L_{2}([0,T],B^{{\alpha}}_{\tau,\infty}(D)) \quad \text{for all}\quad {0<{\alpha<\min\left(\gamma,\frac{3}{\delta}m\right)}}, 
\end{equation}
where $\frac 12 <\frac{1}{\tau}<\frac{\alpha}{3}+\frac 12$ and $\delta$ denotes the dimension of the singular set of $D$. 
In particular, for any ${\alpha}$ satisfying \eqref{parab-Besov2} and $\tau$ as above, we have the a priori estimate 
\begin{align*}
\|u|& L_{2}([0,T],B^{{\alpha}}_{\tau,\infty}(D))\|
\lesssim  
\sum_{k=0}^{\gamma-2m}\|\partial_{t^k} f|{L_2([0,T],{\mathcal{K}^{\gamma-2m}_{2,a}(D))}}\|+\sum_{k=0}^{(\gamma-2m)+1}\|\partial_{t^k} f|{L_2(D_T)}\|. 
\end{align*}
\end{theorem}
}

\begin{proof}
According to Theorem  \ref{thm-weighted-sob-reg-2} by our assumptions we know  
$u\in L_2([0,T], {\mathcal{K}^{\gamma}_{2,a+2m}(D)})$. Together with  Theorem  \ref{thm-hansen-gen} {(choosing $k=0$)} we obtain  
\begin{align*}
u\in &L_2([0,T],{\mathcal{K}^{\gamma}_{2,a+2m}(D)})\cap {{H}}^{m,{1}\ast}(D_T)\\
&\hookrightarrow  L_2([0,T],\mathcal{K}^{\gamma}_{2,a+2m}(D))\cap L_2([0,T],H^m(D))\\
&\hookrightarrow   L_2([0,T],\mathcal{K}^{\gamma}_{2,a+2m}(D))\cap L_2([0,T],B^m_{2,\infty}(D))\\
&\hookrightarrow   L_2([0,T],\mathcal{K}^{\alpha}_{2,a+2m}(D)\cap B^m_{2,\infty}(D))
\hookrightarrow    L_2([0,T],B^{{\alpha}}_{\tau,\infty}(D)),   
\end{align*}
 where {$\alpha \leq \gamma$ in the second to last line. Moreover, } the condition on the parameter '$a$' from Theorem \ref{thm-hansen-gen} yields  
$
m=\min(m,a+2m)>\frac{\delta}{3}\alpha.   
$ 
{Therefore, the upper bound for ${\alpha}$ is  $ {\alpha<\min\left(\gamma,\frac{3}{\delta}m\right)}. $ }
Concerning the restriction on $\tau$, Theorem  \ref{thm-hansen-gen} with $\tau_0=2$ gives 
$
\frac 12<\frac{1}{\tau}<\frac{1}{\tau^{\ast}}=\frac{{\alpha}}{3}+\frac 12.
$ 
This finishes the proof. 
\end{proof}

\remark{
 It might not be obvious at first glance  that Assumption \ref{assumptions} is satisfied with the parameter restrictions  in Theorems \ref{thm-parab-Besov} and \ref{thm-parab-Besov-2}. For a  discussion on this subject we refer to \cite[Rem.~3.8, Ex~4.8]{DS19}, where this matter was discussed in detail and exemplary illustrated for the heat equation. We do not want to repeat the arguments here. 
}

\subsection{Besov regularity of Problem II}\label{Subsect-4.2}

Concerning  the  Besov regularity  of  Problem \ref{prob_nonlin}, we proceed in the same way as before for Problem \ref{prob_parab-1a}: Combining  Theorem \ref{nonlin-B-reg1} (Nonlinear Sobolev and Kondratiev regularity) with the embeddings from Theorem \ref{thm-hansen-gen}  we derive the following result.

\begin{theorem}[{\bf Nonlinear Besov regularity}]\label{nonlin-B-reg3}
Let the assumptions of Theorems \ref{nonlin-B-reg1} and  \ref{thm-weighted-sob-reg} be satisfied.  
In particular, as in Theorem \ref{nonlin-B-reg1} for $\eta:=\|f|\widetilde{\mathcal{D}}\|$ and $r_0>1$,  we choose $\varepsilon >0$ so small that 
\begin{equation}\label{nonlin-cond1}
{
\eta^{2(M-1)} \|\tilde{L}^{-1}\|^{2M-1}\leq \frac{1}{{c}\varepsilon M}(r_0-1)\left(\frac{1}{r_0}\right)^{2M-1}, \qquad \text{if}\quad  r_0\|\tilde{L}^{-1}\|\eta>1,
}
\end{equation}
and 
\begin{equation}\label{nonlin-cond2}
\|\tilde{L}^{-1}\|<\frac{r_0-1}{r_0}\left(\frac{1}{c\varepsilon M}\right), \qquad \text{if}\quad  r_0\|\tilde{L}^{-1}\|\eta<1.
\end{equation}
\begin{figure}[H]
\begin{minipage}{0.55\textwidth}
{\em Then there exists a solution $u$ of Problem \ref{prob_nonlin}, which  satisfies 
 $ u\in B_0\subset B$, 
$$B:=L_2([0,T],B^{\alpha}_{\tau,\infty}(D)), 
$$ 
{\text{for all} $0<\alpha<\min\left(\frac{3}{\delta}m,\gamma\right)$}, where $\delta$ denotes the dimension of the singular set of $D$, $\frac 12<\frac{1}{\tau}<\frac{\alpha}{3}+\frac 12$,  and  $B_0$ is a small ball  around $\tilde{L}^{-1}f$ (the solution of the corresponding linear problem) with radius $R={C\tilde{C}}(r_0-1)\eta \|\tilde{L}^{-1}\|$. }
\end{minipage}\hfill \begin{minipage}{0.38\textwidth}
\includegraphics[width=7.2cm]{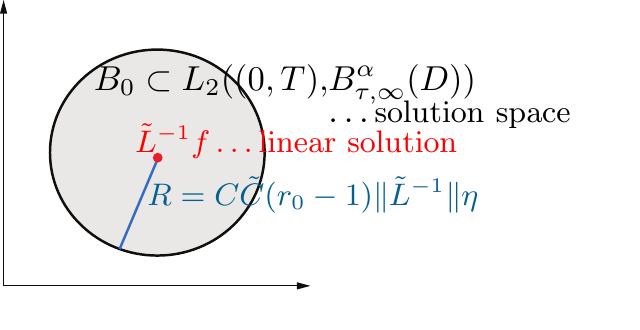}
\caption[Nonlinear solution in  ball]{Nonlinear solution in $B_0$}
\end{minipage}
\end{figure}
\end{theorem}

\begin{proof}
This is a consequence of the regularity results in Kondratiev and Sobolev spaces from Theorem  \ref{nonlin-B-reg1}. To be more precise, Theorem  \ref{nonlin-B-reg1} establishes the existence of a fixed point $u$ in  
\begin{align*}
S_0\subset S
&{:= \bigcap_{k=0}^{\gamma_m+1}H^{k}([0,T],\mathcal{K}^{2m(\gamma_m-(k-1))}_{2,a+2m(\gamma_m-(k-1))}(D)) \cap {H}^{m,\gamma_m+2\ast}(D_T)}\\
&\hookrightarrow  \bigcap_{k=0}^{\gamma_m+1}H^{k}([0,T],\mathcal{K}^{2m(\gamma_m-(k-1))}_{2,a+2m(\gamma_m-(k-1))}(D)) \\
&\qquad \qquad \cap H^{\gamma_m+1}([0,T],{{H}^m}(D))\cap H^{\gamma_{m}+2}([0,T], H^{-m}(D)) \\
& \hookrightarrow L_2([0,T], \mathcal{K}^{2m(\gamma_m+1)}_{2,a+2m(\gamma_m+1)}(K)\cap H^m(D))=:\tilde{S}. 
\end{align*}
This together with the embedding results for Besov spaces from Theorem  \ref{thm-hansen-gen} {(choosing $k=0$)} completes the proof, {in particular, we calculate for the solution (cf. the proof of Theorem \ref{thm-parab-Besov})
\begin{align}
\| u&-\tilde{L}^{-1}f |L_2([0,T],B^{\alpha}_{\tau,\infty}(D))\|\notag \\
&\leq  C \| u- \tilde{L}^{-1}f| L_2([0,T], \mathcal{K}^{2m(\gamma_m+1)}_{2,a+2m(\gamma_m+1)}(D)\cap H^m(D))\|\notag\\
&  = C\|u-\tilde{L}^{-1}f| \tilde{S}\|
\leq  C\tilde{C}\|u-\tilde{L}^{-1}f| S\|
\leq  C\tilde{C}(r_0-1)\eta \|\tilde{L}^{-1}\|. \label{est-abc}
\end{align} Furthermore,  it can be seen from \eqref{est-abc} that new constants $C$ and $\tilde{C}$ appear when considering the radius $R$ around the linear solution where the problem can be solved compared to Theorem \ref{nonlin-B-reg1}.}
\end{proof}

\remark{A few words concerning the parameters appearing in Theorem \ref{nonlin-B-reg3} (and also Theorem  \ref{nonlin-B-reg1}) seem to be in order. 
Usually, the operator norm  $\|\tilde{L}^{-1}\|$ as well as  $\varepsilon$ are fixed; but we can change $\eta$ and $r_0$ according to our needs. From this we deduce that by choosing $\eta$ small enough the  {condition \eqref{nonlin-cond2} can always be satisfied.} Moreover, it is easy to see that the smaller the nonlinear perturbation $\varepsilon>0$ is, the larger we can choose the radius $R$ of the ball $B_0$ where the  solution of Problem \ref{prob_nonlin} is unique. 
}

\subsection{H\"older-Besov regularity of Problem I}
\label{sect-spacetime}

So far we have not exploited the fact that Theorem \ref{thm-weighted-sob-reg} (Kondratiev regularity A)  not only provides regularity {properties}  of the  solution $u$ of Problem \ref{prob_parab-1a} but also of  {its} partial derivatives $\partial_{t^k} u$. We use this fact in combination  with  Theorem \ref{thm-sob-emb}  in order to obtain some  mixed H\"older-Besov regularity results on the whole space-time cylinder $D_T$. \\ 
For parabolic SPDEs, results in this direction have been obtained in \cite{CKLL13}. However, for SPDEs, the time regularity is limited in nature. This is caused by the non-smooth character of the driving processes. Typically, H\"older regularity $\mathcal{C}^{0,\beta}$ can be obtained, but not more. In contrast to this, it is well-known that deterministic parabolic PDEs are smoothing in time. Therefore,   in the deterministic case considered here, higher regularity results in time can be obtained {compared to} the probabilistic setting.   \\

\begin{theorem}[{\bf H\"older-Besov regularity}]\label{Hoelder-Besov-reg}
Let $D$ be a bounded polyhedral domain in $\real^3$. Moreover,  let $\gamma\in \nat $ with  $\gamma\geq 4m+1$ and put $\gamma_m:=\left[ \frac{\gamma-1}{2m}\right]$. Furthermore, let  $a\in \real$ with   ${a\in [-m,m]}$.  Assume that  the right-hand side $f$ of Problem \ref{prob_parab-1a} satisfies 
\begin{itemize}
\item[(i)] $\partial_{t^k} f\in L_2(D_T)\cap L_2([0,T],\mathcal{K}^{2m(\gamma_m-k)}_{2,a+2m(\gamma_m-k)}(D))$, \ $k=0,\ldots, \gamma_m$, \quad 
$\partial_{t^{\gamma_m+1}} f\in L_2(D_T)$.   
\item[(ii)] $\partial_{t^k} f(x,0)=0$, \quad  $k=0,1,\ldots, {\gamma_m}.$
\end{itemize}
Let  Assumption \ref{assumptions}  hold for weight parameters $b=a+2m(\gamma_m-i)$, where $i=0,\ldots, \gamma_m$  and  $b'=-m$.
Then for the  solution $u\in {H}_2^{m,\gamma_m+2\ast}(D_T)$ of Problem \ref{prob_parab-1a}, we have 
$$
u\in \mathcal{C}^{{\gamma_m-2},\frac 12}([0,T],B^{\eta}_{\tau,\infty}(D)) \quad \text{for all}\quad 0<\eta<\min\bigg(\frac{3}{\delta}, 4\bigg)m, 
$$  
where $\frac 12 <\frac{1}{\tau}<\frac{\eta}{3}+\frac 12$ and $\delta$ denotes the dimension of the singular set of $D$.  In particular, we have the a priori estimate 
\begin{align*}
\|u&|\mathcal{C}^{{\gamma_m-2},\frac 12}([0,T],B^{\eta}_{\tau,\infty}(D))\|
\lesssim  {\sum_{k=0}^{\gamma_m}\|\partial_{t^k} f|L_2([0,T], \mathcal{K}^{2m(\gamma_m-k)}_{2,a+2m(\gamma_m-k)}(D))\|+\sum_{k=0}^{\gamma_m+1}\|\partial_{t^k} f|{L_2(D_T)}\|},
\end{align*}
where the constant is independent of $u$ and $f$. 
\end{theorem}

\begin{proof} Theorems \ref{thm-weighted-sob-reg} and  \ref{Sob-reg-3} show together with Theorems  \ref{thm-hansen-gen}  and  \ref{thm-sob-emb}, that under the given assumptions on the initial data $f$, we have {for $k\leq  \gamma_m-2$, }
\begin{eqnarray*}
u &\in & {H^{k+1}([0,T], {\mathcal{K}^{2m(\gamma_m-k)}_{2,a+2m(\gamma_m-k)}(D)})}\cap H^{{\gamma_m}+1}([0,T], H^m(D))\\
&\hookrightarrow & {H^{k+1}([0,T], {\mathcal{K}^{2m(\gamma_m-k)}_{2,a+2m(\gamma_m-k)}(K)}\cap H^m(D))}\\
&\hookrightarrow&  {\mathcal{C}^{{k},\frac 12}([0,T], {\mathcal{K}^{2m(\gamma_m-k)}_{2,a+2m(\gamma_m-k)}(D)}\cap H^m(D)) }\\
&\hookrightarrow &{\mathcal{C}^{{k},\frac 12}([0,T], {\mathcal{K}^{\eta}_{2,a+2m(\gamma_m-k)}(D)}\cap H^m(D)) }
\hookrightarrow  \mathcal{C}^{{k},\frac 12}([0,T], B^{\eta}_{\tau,\infty}(D)), 
\end{eqnarray*}
{where in the third step we require {$\eta\leq 2m(\gamma_m-k)$} and by Theorem \ref{thm-hansen-gen} we get the additional restriction  
$
m=\min(m,a+2m(\gamma_m-k))\geq \frac{\delta}{3}\eta, \  \text{i.e.}, \ \eta <\frac{3}{\delta}m. 
$
Therefore, the upper bound on $\eta$ reads as $\eta<\min\left(\frac{3}{\delta}m, 2m(\gamma_m-k)\right)$ since $k\leq \gamma_m-2$, which for $k=\gamma_m-2$ yields $\eta<\min\left(\frac{3}{\delta}, 4\right)m$.}   
\end{proof}

{
\rem{\hfill 
\begin{itemize}
\item[(i)] For {$\gamma\geq 2m+1$ and} $k=\gamma_m-1$ we have $\eta\leq \min\left(\frac{3}{\delta},2\right)m$ in the theorem above.  For {$\gamma\geq 2m$ and} $k=\gamma_m$ we get $\eta=0$. 
\item[(ii)] From the proof of Theorem \ref{Hoelder-Besov-reg} above it can be seen that the solution satisfies 
\mbox{$u\in {\mathcal{C}^{{k},\frac 12}([0,T], {\mathcal{K}^{2m(\gamma_m-k)}_{2,a+2m(\gamma_m-k)}(D)}), }$}  
implying that for high regularity in time, which is displayed by the parameter $k$, we have less spatial regularity in terms of $2m(\gamma_m-k)$. 
\end{itemize}
}
}

\appendix

\section{Appendix}
\subsection{Preliminaries}
\label{app-not}

We collect some  notation used throughout the paper. As usual,  we denote by $\nat$ the set of all natural numbers, $\nat_0=\mathbb N\cup\{0\}$, and 
$\real^d$, $d\in\nat$,  the $d$-dimensional real Euclidean space with $|x|$, for $x\in\real^d$, denoting the Euclidean norm of $x$. 
By $\mathbb{Z}^d$ we denote the lattice of all points in $\real^d$ with integer components. 
For $a\in\real$, let  
$[a]$ denote its integer part. \\
Moreover,  $c$ stands for a generic positive constant which is independent of the main parameters, but its value may change from line to line. 
The expression $A\lesssim B$ means that $ A \leq c\,B$. If $A \lesssim
B$ and $B\lesssim A$, then we write $A \sim B$.  

Given two quasi-Banach spaces $X$ and $Y$, we write $X\hookrightarrow Y$ if $X\subset Y$ and the natural embedding is bounded. By $\supp f$ we denote the support of the function $f$. For a domain $\Omega\subset \real^d$ and $r\in \nat\cup \{\infty\}$ we write $C^r(\Omega)$ for the space of all {real}-valued $r$-times continuously differentiable functions, 
whereas $C(\Omega)$ is the space of bounded uniformly continuous functions, and  $\mathcal{D}(\Omega)$ for the set of test functions, i.e., the collection of all infinitely differentiable functions with  {compact support contained in $\Omega$. Moreover,  $L^1_{\text{loc}}(\Omega)$ denotes the space of locally integrable functions on $\Omega$.} \\
For  a multi-index  $\alpha = (\alpha_1, \ldots,\alpha_d)\in \nat_0^d$ with  $|\alpha| := \alpha_1+\ldots+ \alpha_d=r$, $r\in \nat_0$,  and an $r$-times differentiable function $u:\Omega\rightarrow \real$, we write 
\[
D^{(\alpha)}u=\frac{\partial^{|\alpha|}}{\partial x_1^{\alpha_1}\dots \partial x_d^{\alpha_d}} u
\]
for the corresponding classical partial derivative as well as $u^{(k)}:=D^{(k)}u$ in the one-dimensional case. Hence, the space $C^r(\Omega)$ is normed by 
\[
\| u| C^r(\Omega)\|:=\max_{|\alpha|\leq r}\sup_{x\in \Omega}|D^{(\alpha)}u(x)|<\infty. 
\]
Moreover, $\mathcal{S}(\real^d)$ denotes the Schwartz space of rapidly decreasing functions. The set of distributions on $\Omega$ will be denoted by $\mathcal{D}'(\Omega)$, whereas $\mathcal{S}'(\real^d)$ denotes the set of tempered distributions on $\real^d$. The terms {\em distribution} and {\em generalized function} will be used synonymously. For the application of a distribution $u\in \mathcal{D}'(\Omega)$ to a test function $\varphi\in \mathcal{D}(\Omega)$ we write $(u,\varphi)$. The same notation will be used if $u\in \mathcal{S}'(\real^d)$ and $\varphi\in \mathcal{S}(\real^d)$ (and also for the inner product in $L_2(\Omega)$).  For $u\in \mathcal{D}'(\Omega)$  and a multi-index $\alpha = (\alpha_1, \ldots,\alpha_d)\in \nat_0^d$, we write $D^{\alpha}u$ for the $\alpha$-th {\em generalized} or {\em distributional derivative} of $u$ with respect to $x=(x_1,\ldots, x_d)\in \Omega$, i.e., $D^{\alpha}u$ is a distribution on $\Omega$, uniquely determined by the formula   
\[
(D^{\alpha}u,\varphi):=(-1)^{|\alpha|}(u,D^{(\alpha)}\varphi), \qquad \varphi \in \mathcal{D}(\Omega). 
\]
{In particular, if  $u\in L^1_{\text{loc}}(\Omega)$ and  there exists a function $v\in L^1_{\text{loc}}(\Omega)$ such that 
\[
\int_\Omega v(x)\varphi(x)\ud x=(-1)^{|\alpha|}\int_{\Omega}u(x)D^{(\alpha)}\varphi(x)\ud x \qquad \text{for all} \qquad \varphi \in \mathcal{D}(\Omega), 
\]
we say that $v$ is the {\em $\alpha$-th weak derivative} of $u$ and  write $D^{\alpha}u=v$. 
}
We also use the notation $
\frac{\partial^k}{\partial x_j^k}u:=D^{\beta}u
$ as well as $\partial_{x_j^k}:=D^{\beta}u$,   for some 
multi-index  $\beta=(0,\ldots, k, \ldots,0)$ with $\beta_j=k$, $k\in \nat$. Furthermore, for $m\in \nat_0$, we write $D^mu$ for any (generalized) $m$-th order derivative of $u$, where $D^0u:=u$ and $Du:=D^1u$. Sometimes we shall use subscripts such as $D^m_x$ or  $D^m_t $ to emphasize that we only take derivatives with respect to $x=(x_1, \ldots, x_d)\in \Omega$ or $t\in \real$.

\subsection{Besov spaces}
\label{sect-Besov}

Due to the different contexts Besov spaces arose from they can be defined/characterized in several ways, e.g. via higher order differences, the Fourier-analytic approach or  decompositions with suitable building blocks, cf. \cite{Tri83, Tri08} and the references therein. Under certain restrictions on the parameters these different approaches might even coincide. Throughout this paper  we rely on the characterization of Besov spaces via  wavelet decompositions  and refer in this context to 
\cite{Coh03, Mey92}.  Let us briefly recall the concept: 
Wavelets are specific orthonormal bases for $L_2(\mathbb{R})$ that are obtained by dilating, translating and scaling one fixed function, the so--called  {\em mother wavelet} $\psi$. The mother wavelet is usually constructed by means of a so-called {\em multiresolution analysis,} that is, a sequence  $\{V_j\}_{j \in \mathbb{Z}}$  of shift-invariant, closed subspaces of $L_2(\mathbb{R})$ whose union is dense in $L_2$ while their intersection is trivial. Moreover, all the spaces are related via dyadic dilation, and the space  $V_0$  is spanned  by the translates of  one fixed function $\phi$, called the {\em generator}.  In her fundamental work \cite{Dau1, Dau2}  I. Daubechies has shown that there exist families of compactly supported wavelets. By taking tensor products, a compactly supported orthonormal basis for $L_2(\mathbb{R}^d)$ can be constructed.\\
Let ${\phi}$ be a father wavelet of tensor product type on $\real^d$ and let $\Psi'=\{\psi_i: \ i=1,\ldots, 2^d-1\}$ be the set containing the corresponding multivariate mother wavelets such that, for a given $r\in \nat$ and some $N>0$ the following localization, smoothness and vanishing moment conditions hold. For all $\psi\in \Psi'$, 
\begin{align}
\supp{\phi}, \ \supp \psi   & \ \subset \ [-N,N]^d, \label{wavelet-1}\\
{\phi}, \ \psi  & \ \in \ C^r(\real^d), \label{wavelet-2}\\
\int_{\real^d} x^{\alpha}\psi(x)\ud x & \ =\ 0 \quad \text{ for all } \alpha \in \nat_0^d \ \text{ with } \  \ |\alpha|\leq r. \label{wavelet-3}
\end{align}  
We refer again to  \cite{Dau1, Dau2} for a detailed discussion. 
The set of all dyadic cubes in $\real^d$ with measure at most $1$ is denoted by 
\[
\mathcal{D}^{+}:=\left\{I\subset \real^d: \ I=2^{-j}([0,1]^d+k), \ j\in \nat_0, \ k\in \mathbb{Z}^d\right\}
\]
and we set $\mathcal{D}_j:=\{I\in \mathcal{D}^+: \ |I|=2^{-jd}\}.$ 
For the dyadic shifts and dilations of the father wavelet and the corresponding wavelets we use the abbreviations 
\begin{equation}\label{wavelet-4}
{\phi}_k(x):={\phi}(x-k), \quad \psi_{I}(x):=2^{jd/2}\psi(2^jx-k) \qquad \text{for}\quad  j\in \nat_0, \ k\in \mathbb{Z}^d, \ \psi\in \Psi'. 
\end{equation}
It follows that 
\[
\left\{{\phi}_k, \ \psi_{I}:  \ k\in \mathbb{Z}^d, \  I\in \mathcal{D}^+, \  \psi\in \Psi'\right\}
\]
is an orthonormal basis in $L_2(\real^d)$. 
Denote by $Q(I)$ some dyadic cube (of minimal size) such that $\supp \psi_I \subset Q(I)$ for every $\psi\in \Psi'$. Then, we clearly have $Q(I)=2^{-j}k+2^{-j}Q$ for some dyadic cube $Q$. Put $\Lambda'=\mathcal{D}^{+}\times \Psi'$.  
Then, every function $f\in L_2(\real^d)$ can be written as 
\[
f=
\sum_{k\in \mathbb{Z}^d}\langle f,{{\phi}}_k\rangle {{\phi}}_k +\sum_{(I,\psi)\in \Lambda'}\langle f, {\psi}_I\rangle \psi_I.  
\]
It will be convenient to include ${\phi}$ into the set $\Psi'$. We use the notation ${\phi}_I:=0$ for $|I|<1$, ${\phi}_I={\phi}(\cdot-k)$ for $I=k+[0,1]^d$, and can simply write 
\[
f=\sum_{(I,\psi)\in \Lambda}\langle f, {\psi}_I\rangle \psi_I, \qquad \Lambda=\mathcal{D}^+\times \Psi, \quad \Psi=\Psi'\cup \{{\phi}\}.
\]

We describe Besov spaces on $\real^d$ by decay properties of the wavelet coefficients, if the parameters fulfill certain conditions.  \\

\begin{theorem}[{Wavelet decomposition of Besov spaces}]\label{thm-wavelet-dec}
Let $0<p,q<\infty$ and $s>\max\left\{0,d(1/p-1)\right\}$. Choose $r\in \nat$ such that $r>s$ and construct a wavelet Riesz basis as described above. Then a function $f\in L_p(\real^d)$ belongs to the Besov space $B^s_{p,q}(\real^d)$ if, and only if, 
\begin{equation}\label{besov-decomp}
f=\sum_{k\in \mathbb{Z}^d}\langle f,{{\phi}}_k\rangle {\phi}_k +\sum_{(I,\psi)\in \Lambda'}\langle f, {\psi}_I\rangle \psi_I  
\end{equation}
(convergence in $\mathcal{S}'(\real^d)$) with 
\begin{align}
\|f|B^s_{p,q}(\real^d)\| 
&\sim  \left(\sum_{k\in \mathbb{Z}^d} |\langle f,{{\phi}}_k\rangle|^p\right)^{1/p} + \notag\\
& \qquad   \left(\sum_{j=0}^{\infty}2^{j\left(s+d(\frac 12-\frac 1p)\right)q}\left(\sum_{(I,\psi)\in \mathcal{D}_j\times \Psi'}|\langle f, {\psi}_{I}\rangle|^p\right)^{q/p}\right)^{1/q}<\infty.\label{besov-norm}
\end{align}
\end{theorem}

{
\begin{remark}{
 In particular, for our adaptivity scale \eqref{adaptivityscale}, i.e., $B^s_{\tau,\tau}(\real^d)$ with $s=d\left(\frac{1}{\tau}-\frac 1p\right)$,  we see that the quasi-norm  \eqref{besov-norm} becomes 
\begin{align}
\|f|B^s_{\tau,\tau}(\real^d)\|&\sim  \left(\sum_{k\in \mathbb{Z}^d} |\langle f,{{\phi}}_k\rangle|^{\tau}\right)^{1/\tau} +   
   \left(\sum_{j=0}^{\infty}2^{jd\left(\frac 12-\frac 1p\right)\tau}\sum_{(I,\psi)\in \mathcal{D}_j\times \Psi'}|\langle f, {\psi}_{I}\rangle|^{\tau}\right)^{1/\tau}. \notag\label{besov-norm2}
\end{align}
}
\end{remark}
}

Corresponding function spaces on domains $\mathcal{O}\subset \real^d$ can be introduced via restriction, i.e., 
\begin{eqnarray*}
B^s_{p,q}(\mathcal{O})&=& \left\{f\in \mathcal{D}'(\mathcal{O}): \ \exists g\in B^s_{p,q}(\real^d), \ g\big|_{\mathcal{O}}=f \right\},\\
\|f|B^s_{p,q}(\mathcal{O})\|&=& \inf_{g|_{\mathcal{O}}=f}\|f|B^s_{p,q}(\real^d)\|. 
\end{eqnarray*}
Alternative (different or equivalent) versions of this definition can be found, depending on possible additional properties of the distributions $g$ (most often their support). We refer to \cite{Tri08} for details and references. 

\remark{ 
We remark that  the Besov (and Kondratiev) spaces we are working with are defined in the setting of distributions, i.e., as subsets of $\mathcal{D}'(\mathcal{O})$,  and therefore may contain 'functions' which take complex values. However, when considering the fundamental parabolic problems, we restrict ourselves to the  real-valued setting: We assume the  coefficients of the differential operator $L$ to be real-valued as well as the right-hand side $f$, therefore, the solutions are real-valued as well. 
}


\begin{thebibliography}{m}



\bibitem{ADN59}
Agmon, S.,   Douglis, A.,  Nierenberg, L. (1959).   
\newblock Estimates near the boundary for solutions of elliptic partial differential equations satisfying general boundary conditions I.  
\newblock{\em Comm. Pure Appl. Math.} {\bf 12}, 623--727.  

\bibitem{AGI08}
Aimar, H.,  G\'omez, I., Iaffei, B. (2008). 
\newblock {Parabolic mean values and maximal estimates for gradients of temperatures.}
\newblock {\em J. Funct. Anal.} {\bf 255}, 1939--1956.

\bibitem{AGI10}
Aimar, H.,  G\'omez, I., Iaffei, B. (2010).
\newblock {On Besov regularity of temperatures.}
\newblock {\em J. Fourier Anal. Appl.} {\bf 16}, 1007--1020.

\bibitem{AG12}
Aimar, H., G\'omez, I. (2012).
\newblock {Parabolic Besov Regularity for the Heat Equation.}
\newblock {\em Constr. Approx.} {\bf 36}, 145--159.

\bibitem{AH08}
 Anh, N. T. and Hung, N. M. (2008).
\newblock Regularity of solutions of initial-boundary value problems for parabolic equations in domains with conical points.
\newblock {\em J. Differential Equations} {\bf 245},  no. 7, 1801--1818.

\bibitem{ALL16}
 Anh, N. T., Loi, D. V., and Luong, V. T. (2016).
\newblock $L_p$-regularity for the Cauchy-Dirichlet problem for parabolic equations in convex polyhedral domains.
\newblock {\em Acta Math. Vietnam.} {\bf 41},  no. 4, 731--742.

\bibitem{BG97}
Babuska, I., Guo, B. (1997).
\newblock Regularity of the solutions for elliptic problems on nonsmooth domains in $\real^3$, Part I: countably normed spaces on polyhedral domains.
\newblock {\em Proc. Roy. Soc. Edinburgh Sect. A}, {\bf 127}, 77--126.


\bibitem{BMNZ}
{Bacuta, C.,  Mazzucato, A., Nistor, V., and  Zikatanov, L.} (2010). 	
{Interface and mixed boundary value problems on $n$-dimensional polyhedral domains}.
{\em  Doc. Math.} {\bf 15}, 687--745.




\bibitem{CKLL13}
Cioica, P., Kim, K.-H., Lee, K., Lindner, F. (2013).
\newblock {On the $L_q(L_p)$-regularity and Besov smoothness of stochastic parabolic equations on bounded Lipschitz domains.}
\newblock {\em Electron. J. Probab.} {\bf 18}(82), 1--41.


\bibitem{CW20}
Cioica-Licht, P. and Weimar, M. (2020).
\newblock {On the limit regularity in Sobolev and Besov scales related to approximation theory.}
\newblock {\em J. Fourier Anal. Appl.} {\bf 26}(1), 10.

\bibitem{Coh03}
Cohen, A. (2003).
\newblock{\em Numerical Analysis of wavelet Methods}. 
\newblock{ Studies in Mathematics and its Applications}, 1st edition, vol. 32, Elsevier, Amsterdam.  

\bibitem{Cost19}
Costabel, M. (2019).
\newblock {On the limit Sobolev regularity for Dirichlet and Neumann problems on Lipschitz domains.}
\newblock {\em Math. Nachr.} {\bf 292}, 2165--2173.



 \bibitem{Dah98} 
Dahlke, S. (1998).  
\newblock Besov regularity for elliptic boundary value problems with variable coefficients. 
\newblock {\em Manuscripta Math.} {\bf 95}, 59--77.



\bibitem{Dah99a}
Dahlke, S. (1999).  
\newblock Besov regularity for interface problems. 
\newblock {\em Z. Angew. Math. Mech.} {\bf 79}(6),  
       383--388.     
       
 \bibitem{Dah99b}
Dahlke, S. (1999).  
\newblock Besov regularity for elliptic boundary value problems on polygonal domains.  
\newblock {\em Appl. Math.  Lett.} {\bf 12}(6),  31--38.

\bibitem{Dah02} 
Dahlke, S. (2002).  
\newblock Besov regularity of edge singularities for the Poisson equation in polyhedral
domains.  
\newblock {\em Num. Linear Algebra  Appl.} {\bf 9}(6--7), 457--466. 


\bibitem{DDD}
Dahlke, S., Dahmen, W.,  DeVore, R. (1997). 
\newblock {Nonlinear approximation and adaptive techniques for solving elliptic operator equations}.  
\newblock {\em Multiscale Wavelet Methods for Partial Differential Equations, (W. Dahmen, A.J. Kurdila, and P. Oswald, eds), Wavelet Analysis and Applications}, vol. 6, Academic Press, San Diego,  237-283.  



\bibitem{DDV97}
Dahlke, S.,  DeVore, R. (1997). 
\newblock Besov regularity for elliptic boundary value problems. 
\newblock {\em Comm.  Partial Differential Equations} {\bf 22}(1-2),   1--16. 


\bibitem{DDHSW}
Dahlke, S., Diening, L., Hartmann, C., Scharf, B., Weimar, M. (2016).  
\newblock Besov regularity of solutions to the p--Poisson equation.  
\newblock {\em Nonlinear Anal.} {\bf 130},  298--329, 2016. 



\bibitem{DHS17a}
Dahlke, S., Hansen, M., Schneider, C., Sickel, W. (2018). 
\newblock Properties of Kondratiev spaces.
\newblock {\em Preprint-Reihe Philipps University Marburg}, Bericht Mathematik Nr. 2018-06. 
\newblock { (arXiv:1911.01962)}


\bibitem{DS18}
Dahlke, S., Schneider, C. (2018). 
\newblock Describing the singular behaviour of parabolic equations on cones in fractional Sobolev spaces.
\newblock {\em Int. J. Geomath.},  {\bf 9}(2), 293--315.




\bibitem{DS19}
Dahlke, S. and Schneider, C. (2019).   
\newblock Besov regularity of parabolic and hyperbolic PDEs. 
\newblock {\em Anal. Appl.} {\bf 17}, no. 2,  235--291. 


\bibitem{DNS2} 
 Dahlke, S., Novak, E., and  Sickel, W. (2006).  
 Optimal approximation of elliptic problems by linear and nonlinear mappings. II. 
 \emph{J.\ Complexity} \textbf{22}, 549--603.



\bibitem{Dau1} 
 Daubechies, I. (1998).
 \newblock Orthonormal bases of compactly supported wavelets. 
 \newblock {\em Comm. Pure Appl. Math.}, {\bf 41}(7), 909--996. 

\bibitem{Dau2} 
 Daubechies, I. (1992).
 \newblock {\em Ten lectures on wavelets}. 
 \newblock {CBMS-NSF Regional Conference Series in Applied Mathematics}, {\bf 61}, SIAM, Philladelphia. 



\bibitem{Eva10}
Evans, L. C.  (2010). 
\newblock Partial differential equations.
\newblock {\em Graduate Studies in Mathematics} {\bf 19}, 2nd edition, American Mathematical Society, Providence, RI. 


\bibitem{GM09}
Gaspoz, F.D., Morin, P. (2009).
\newblock Convergence rates for adaptive finite elements.
\newblock {\em IMA J. Numer. Anal.} {\bf 29}(4), 917--936.

\bibitem{Gris92}
Grisvard, P. (1992).
\newblock {\em Singularities in boundary value problems.}
\newblock {Recherches en math\'ematiques appliqu\'ees}, vol. 22, Masson, Paris;  Springer-Verlag, Berlin. 


\bibitem{Gris11}
Grisvard, P. (2011).
\newblock {\em Elliptic problems in nonsmooth domains.}
\newblock Reprint of the 1985 original. { Classics in Applied Mathematics}, 69, SIAM, Philadelphia. 



\bibitem{Han15}
Hansen, M. (2015).
\newblock Nonlinear approximation rates and Besov regularity for elliptic PDEs on polyhedral domains.
\newblock {\em Found. Comput. Math.} {\bf 15}, 561--589.

\bibitem{HW18}
Hartmann, C. and Weimar, M. (2018).
\newblock Besov regularity of solutions to the $p$-Poisson equation in the vicinity of a vertex of a polygonal domain.
\newblock {\em Results Math.} {\bf 73}(41), 1--28.


\bibitem{JK95}
Jerison, D., Kenig, C.E. (1995).
\newblock The inhomogeneous Dirichlet problem in Lipschitz domains.
\newblock {\em J. Funct. Anal.} {\bf 130}, 161--219. 



\bibitem{KMR01}
Kozlov, V.A., Mazya, V.G., and Rossman, J. (2001).
\newblock Spectral problems associated with corner singularities of solutions to elliptic equations.  
\newblock {\em Mathematical Surveys and Monographs},  {\bf 85}, American Mathematical Society, Providence, RI. 



\bibitem{Lan01}
Lang, J. (2001). 
\newblock {\em Adaptive multilevel solution of nonlinear parabolic {PDE} systems}. 
\newblock {Lecture Notes in Computational Science and Engineering; Theory, algorithm, and applications}, vol. 16, Springer-Verlag, Berlin. 



\bibitem{LL15}
Luong, V.T., Loi, D.V. (2015).
\newblock The first initial-boundary value problem for parabolic equations in a cone with edges.
\newblock {\em Vestn. St.-Petersbg. Univ. Ser. 1. Mat. Mekh. Asron.  2}, {\bf 60}(3), 394--404.



\bibitem{MR10}
Maz'ya, V.,  Rossmann, J. (2010).
\newblock{\em Elliptic equations in polyhedral domains}. 
\newblock{Mathematical Surveys and Monographs}, vol. 162, American Mathematical Society. 


\bibitem{NistorMazzucato}
{Mazzucato, A. and  Nistor, V.} (2010). 
\newblock {Well posedness and regularity for the elasticity equation with mixed boundary conditions on
			polyhedral domains and domains with cracks}. 
\newblock{\em  Arch. Ration. Mech. Anal.} {\bf 195}, 25--73.

\bibitem{Mey92}
Meyer, Y. (1992).
\newblock{\em Wavelets and operators}. 
\newblock{Cambridge Studies in Advances Mathematics}, vol. 37, Cambridge. 

\bibitem{RS96}
Runst, T., Sickel, W. (1996).
\newblock {\em Sobolev spaces of fractional order, {N}emytskij operators, and
              nonlinear partial differential equations}.
\newblock {De Gruyter Series in Nonlinear Analysis and Applications}. 


\bibitem{co-habil}
 Schneider, C. (2020). 
 \newblock  Besov regularity of partial differential equations, and traces in function spaces. 
 \newblock {\em Habilitationsschrift}, Philipps-Universit\"at Marburg. 




\bibitem{SS09} 
Stevenson, R.,  Schwab, C. (2009). 
\newblock Space-time adaptive wavelet methods for parabolic evolution problems.  
\newblock {\em Math. Comput.},  {\bf 78}, 1293--1318. 


\bibitem{Tho06}
Thom\'e{}e, V. (2006). 
\newblock {\em Galerkin finite element methods for parabolic problems}. 
\newblock  {
Springer Series in Computational Mathematics}, vol. 25, Springer-Verlag, Berlin, 2nd edition. 

\bibitem{Tri83}
Triebel, H. (1983).
\newblock {\em Theory of Function spaces}.
\newblock { Monographs in Mathematics}, vol. 78, Birkh\"auser Verlag, Basel. 


\bibitem{Tri08}
Triebel, H. (2008).
\newblock {\em Function spaces and wavelets on domains}.
\newblock {EMS Tracts on mathematics 7}, EMS Publishing House, Z\"urich. 



\bibitem{Wo07}
Wood, I. (2007).
\newblock Maximal $L^p$-regularity for the Laplacian on Lipschitz domains.
\newblock {\em Math. Z.}, {\bf 255}(4), 855--875. 

\end{thebibliography}
\end{document}